\documentclass[10pt,a4paper,reqno]{amsart}

\usepackage{lmodern}
\usepackage[T1]{fontenc}
\usepackage[utf8]{inputenc}
\usepackage[english]{babel}
\usepackage{microtype} 

\usepackage{amsthm,amsmath,amssymb,latexsym,mathtools,enumerate,caption,subcaption}
\usepackage{epigraph,graphicx}
\usepackage[pdftex]{hyperref}
\usepackage[capitalize]{cleveref}

\DeclareGraphicsExtensions{.png,.pdf}

\usepackage{todonotes}
\presetkeys%
    {todonotes}%
    {inline,backgroundcolor=yellow}{}

\newcommand{\defin}[1]{%
\relax\ifmmode%
\textcolor{blue}{#1}%
\else\textcolor{blue}{\emph{#1}}%
\fi%
}

\newtheorem{theorem}{Theorem}[section]
\newtheorem{proposition}[theorem]{Proposition}
\newtheorem{lemma}[theorem]{Lemma}
\newtheorem{corollary}[theorem]{Corollary}
\newtheorem{conjecture}[theorem]{Conjecture}
\newtheorem{problem}[theorem]{Problem}

\theoremstyle{definition}
\newtheorem{definition}[theorem]{Definition}
\newtheorem{example}[theorem]{Example}
\newtheorem{remark}[theorem]{Remark}

\newtheorem{THEO}{Theorem}

\newcommand{\bC}{\mathbb{C}}

\newcommand{\bN}{\mathbb{N}}
\newcommand{\cP}{\mathcal{P}}

\renewcommand{\Im}{{\mathrm{Im}}}

\newcommand{\Z}{\mathcal{Z}}

\newcommand{\zvec}{\mathbf{z}}

\newcommand{\al}{\alpha}
\newcommand{\be}{\beta}

\newcommand{\la}{\lambda}

\newcommand{\F}{\mathcal F}
\newcommand{\Ga}{\Gamma}

\newcommand{\invset}[1]{{{\mathcal I}_{#1}^T}}
\newcommand{\minvset}[1]{{{\mathrm M}_{#1}^T}}
\newcommand{\minvsetT}[2]{{{\mathrm M}_{#1}^{#2}}}

\newcommand{\fidx}{\rho}

\numberwithin{equation}{section}

\title[An inverse problem in Pólya--Schur theory. II.]
{An inverse problem in Pólya--Schur theory. II. Exactly solvable operators and complex dynamics}

\author[P. Alexandersson]{Per Alexandersson}
\address{Department of Mathematics,
Stockholm University, S-10691, Stockholm, Sweden}
\email{per.w.alexandersson@gmail.com}

\author[N. Hemmingsson]{Nils Hemmingsson}
\address{Department of Mathematics,
Stockholm University,
S-10691, Stockholm, Sweden}
\email{nils.hemmingsson@math.su.se}

\author[B. Shapiro]{Boris Shapiro}
\address{Department of Mathematics,
Stockholm University,
S-10691, Stockholm, Sweden}
\email{shapiro@math.su.se}

\begin{document}


\begin{abstract}
This paper, being the sequel of \cite{AlBrSh1}, studies a class of 
linear ordinary differential operators with polynomial coefficients called 
\emph{exactly solvable};  such an operator sends every polynomial of  sufficiently large degree to a polynomial of the same degree. 
We focus on invariant subsets of the complex plane for such operators when their action 
is restricted to polynomials of a fixed degree and discover a connection between this topic and classical complex dynamics and its multi-valued counterpart.  
As a very special case of invariant sets we recover  the Julia sets of rational functions.   
\end{abstract}
\keywords{P\'olya--Schur theory, action of linear differential operators on polynomials, 
 (minimal) $T_n$-invariant sets}
\subjclass[2020]{Primary 37F10; Secondary 34A30}

\maketitle

{\emph{\small 
Science is built up of facts, as a house is with stones. But a collection of facts is no more a science than a heap of stones is a house.}

{\hskip6cm \small   H.~Poincaré, Science and Hypothesis}


\section{Introduction} \label{sec:intro}

To make this text independent of its first part \cite{AlBrSh1}, let us briefly repeat several basic definitions from loc.\ cit. 

\medskip
In 1914, generalizing some earlier results of E.~Laguerre, G.~Pólya and I.~Schur 
created  what later became known as the {\emph{Pólya--Schur theory}}, see~\cite{PolyaSchur1914}.   
%
%
The main question of this  theory can be formulated as follows, see \cite{CravenCsordas2004}. 

\begin{problem}\label{prob1} 
Given a subset $S\subseteq \bC$ of the complex plane, describe the
semigroup 
of  all linear operators $T:\bC[z]\to\bC[z]$
sending any polynomial with all roots in $S$ to a
polynomial with all roots in $S$ (or to the identically zero polynomial).
\end{problem}

\begin{definition}\label{def0}
Given a subset $S\subseteq \bC$ of the complex plane, if an operator $T$ solves Problem~\ref{prob1}, then we say that
$S$ is a \defin{$T$-invariant set}, or that \defin{$T$ preserves $S$}.
\end{definition}


\medskip
In \cite{AlBrSh1}, we  initiated the study of 
the following natural inverse problem in the set-up of the Pólya--Schur theory. (For more information on the Polya--Schur theory consult, e.g., \cite{BorceaBranden2010}).

\medskip
\begin{problem}\label{prob:main}
Given a linear operator $T:\bC[x]\to\bC[x]$, characterize all non-trivial $T$-invariant subsets. Alternatively,  find a (sufficiently large) class of $T$-invariant sets in the complex plane.

\end{problem} 

For example, if $T=\frac{d}{dx}$, then a subset $S\subseteq \bC$ is $T$-invariant if and only if it is convex. 
Although it is too optimistic  to hope for a complete  solution of \cref{prob:main}, in  \cite{AlBrSh1}
  we were able to obtain a number of relevant results  valid
for arbitrary linear differential operators of finite order with polynomial coefficients.  
In particular, among other results, we proved the existence and uniqueness 
of the closed and minimal under inclusion $T$-invariant set for linear differential operator $T$ whose leading coefficient is not a constant. 
In what follows we need the following important notion.

\begin{definition}\label{defFuchs}
Given a linear ordinary differential operator 
\begin{equation}\label{eq:main}
T=\sum_{j=0}^kQ_j(x)\frac{d^j}{dx^j}
\end{equation} 
of order $k\ge 1$  with polynomial coefficients, define its \defin{Fuchs index} as 
\[
\fidx = \max_{0\le j\le k} (\deg Q_{j}-j).
\]
$T$ is called \defin{non-degenerate} if $\deg Q_k-k=\fidx$, and \defin{degenerate} otherwise.
In other words, $T$ is non-degenerate if $\fidx$ is realized by the leading coefficient of $T$. 

\smallskip
We say that $T$ is \defin{exactly solvable} if its Fuchs index is zero.
\end{definition}

This paper is mainly devoted to the study of exactly solvable operators and types of  invariant sets other than those in Definition~\ref{def0}. 
The main reason for our special attention to  this class of operators is that exactly solvable operators can 
be alternatively characterized as ``degree preserving'', see \eqref{eq:spectrum} and that the invariant sets
we consider below resemble those appearing in complex dynamics. 
In particular, they have very different properties than those in Definition~\ref{def0} which were studied in \cite{AlBrSh1}.

\begin{definition}\label{def:degreen}
Given a linear operator $T:\bC[x]\to \bC[x]$, we denote by \defin{$\invset{n}$} the collection  of all closed  subsets $S\subseteq \bC$ 
such that for every polynomial of degree $n$ with all roots in $S$,
its image  $T(p)$ is either $0$  or has all roots in $S$. 
In this situation, we say that $S$  \defin{belongs to the class $\invset{n}$} or that $S$ is \defin{$T_n$-invariant}.
(A constant non-zero polynomial has an empty set of zeros which, by definition, belongs to any set.) 

\smallskip
We say that a closed nonempty set $S \in \invset{n}$  is \defin{minimal} 
if there is no closed proper nonempty subset of $S$ belonging to $\invset{n}$. 
Observe that families of sets belonging to $\invset{n}$ are closed under  taking the intersection. 
Denote by  $\minvset{n}$  the intersection of all sets in $\invset{n}$, provided that this set is non-empty. 
In this case, we say that $\minvset{n}$ exists.
\end{definition}

We will study $\minvset{n}$ in \cref{sec:general}. In particular, we show the existence of $\minvset{n}$ 
under certain weak assumptions on $T$, give its alternative characterization, and study the
asymptotic behavior of the sequence $\{\minvset{n}\}$ when $n\to \infty$. 
To describe this asymptotic behavior we need the following notion. 
 
 \begin{definition} 
Given an operator \eqref{eq:main} with $Q_k(x)\neq const$, denote by $Conv (Q_k)\subseteq \bC$ the convex hull 
of the zero locus of  $Q_k(x)$. 
We will refer to $Conv (Q_k)$ as the \defin{fundamental polygon} of $T$.
\end{definition}

In \cref{sec:general} we will prove that 
the sequence $\{\minvset{n}\}$ ``approaches''  $Conv (Q_k)$ under certain assumptions on $T$, see \cref{fig1}. 

\begin{figure}
\centering
\begin{subfigure}[b]{0.3\textwidth}
\includegraphics[width=\textwidth]{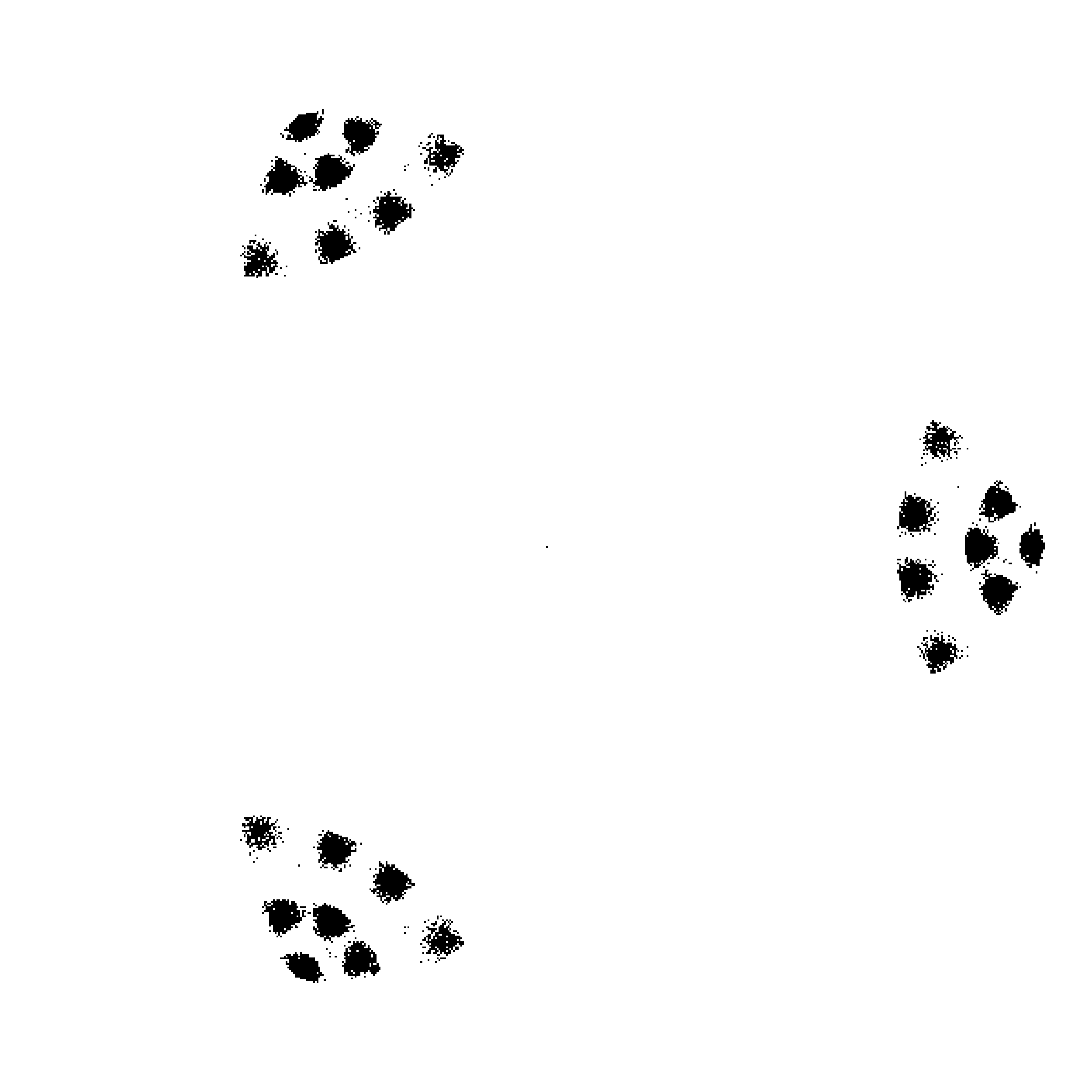}
\caption{$n=3$}
\end{subfigure}%
\begin{subfigure}[b]{0.3\textwidth}
\includegraphics[width=\textwidth]{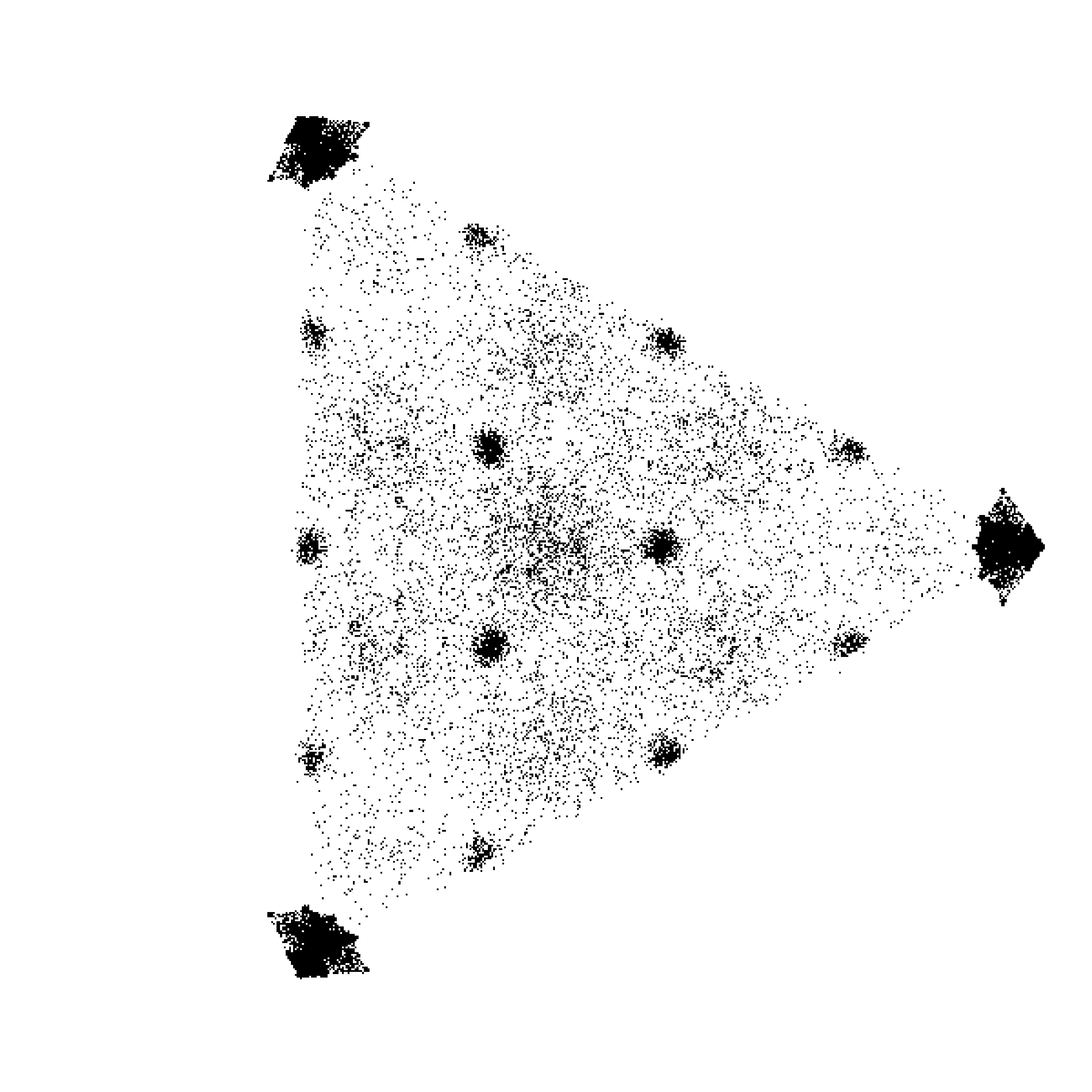}
\caption{$n=4$}
\end{subfigure}%
\begin{subfigure}[b]{0.3\textwidth}
\includegraphics[width=\textwidth]{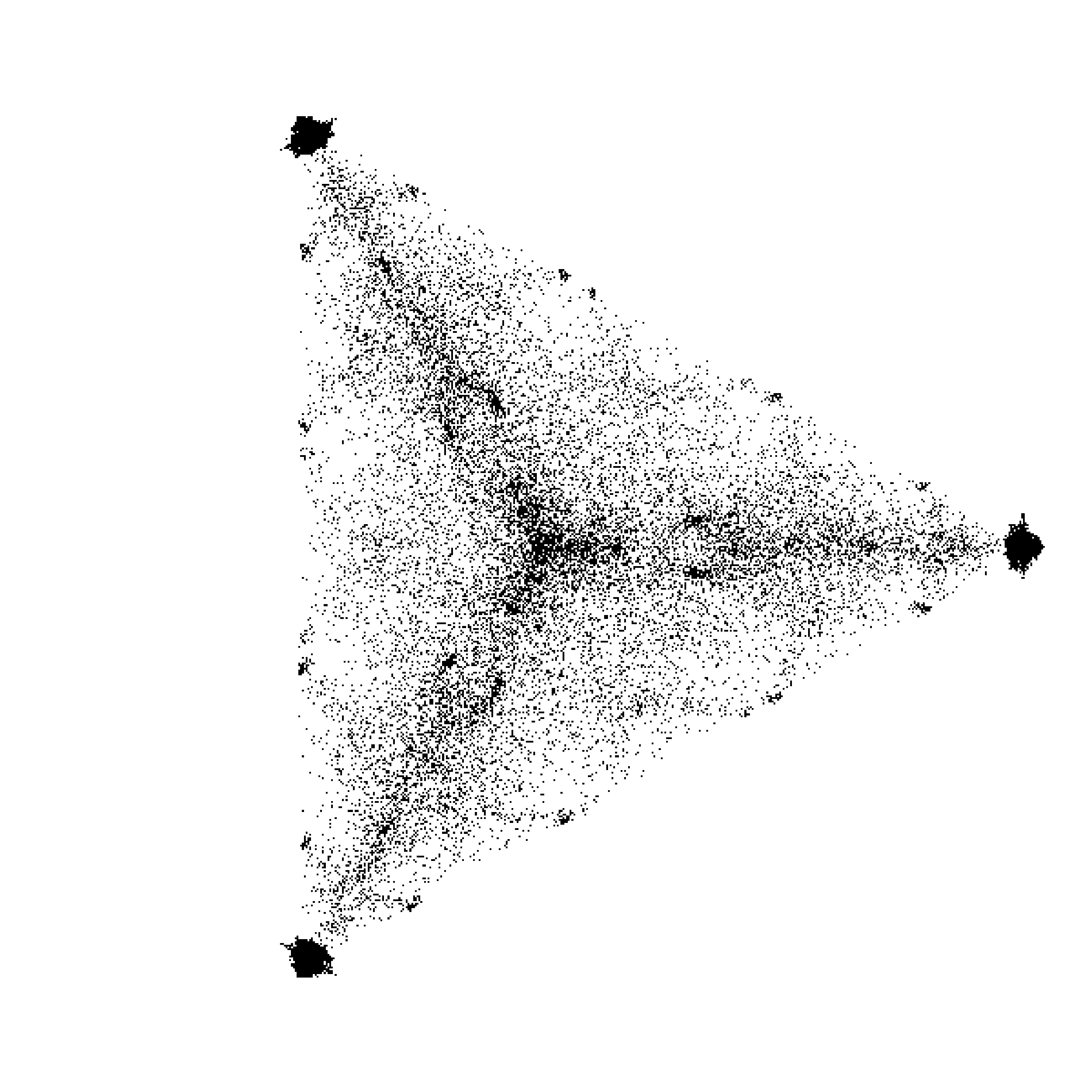}
\caption{$n=5$}
\end{subfigure}%
\caption{The sets $\minvset{n}$ for $T(p)= (x^3-1)p^{\prime\prime\prime} + x p'$.
As $n$ grows, the sequence $\minvset{n}$ approaches the convex hull of the roots of $x^3-1=0$, see \cref{th:M_nConv}. For the existence of $\minvset{n}$, see \cref{th:generalN}. 
}\label{fig1}
\end{figure}

In the context of exactly solvable operators (and  for a more general class of operators \eqref{eq:main}) it is
also natural to introduce one more concept of invariant sets, comp.~\cite{AHNST}. 

\begin{definition}\label{def7}
Given a linear operator $T:\bC[x]\to \bC[x]$, we say that a closed nonempty subset $S\subseteq \bC$ is \defin{Hutchinson-invariant in degree $n$}
if whenever $z\in S$, then $T[(x-z)^n]$ either is $0$ or has all roots in $S$. 
The intersection of all closed Hutchinson-invariant sets is denoted $\minvset{H,n}$, provided it is non-empty. 
In this case, we say that $\minvset{H,n}$ exists.
\end{definition}

The difference between the Definitions~\ref{def:degreen} and~\ref{def7}  is that in the latter case we  only require  $S$ to be 
invariant with respect to  the special type of polynomials of degree $n$, namely, those having an $n$-tuple root,  
rather than all polynomials of degree $n$ with roots in $S$.
It is then straightforward  that $\minvset{H,n} \subseteq \minvset{n}$.

\medskip
We show that the notion $\minvset{H,n}$ generalizes that of the minimal invariant set 
for  the  iterated functions systems studied in \cite{Hutchinson1981} and that  $\minvset{H,n}$  exists under certain assumptions on $T$. 
 We observe that the set $\minvset{H,n}$, in general, seems to be fractal in nature and,  in particular, we will 
 show that the classical fractals such as the Sierpi{\'{n}}ski triangle, the Koch curve and the L{\'e}vy curve (see Figure~\ref{fig:levy})
can be realized as such.

\begin{figure}[!ht]
\centering
  \includegraphics[width=0.5\textwidth]{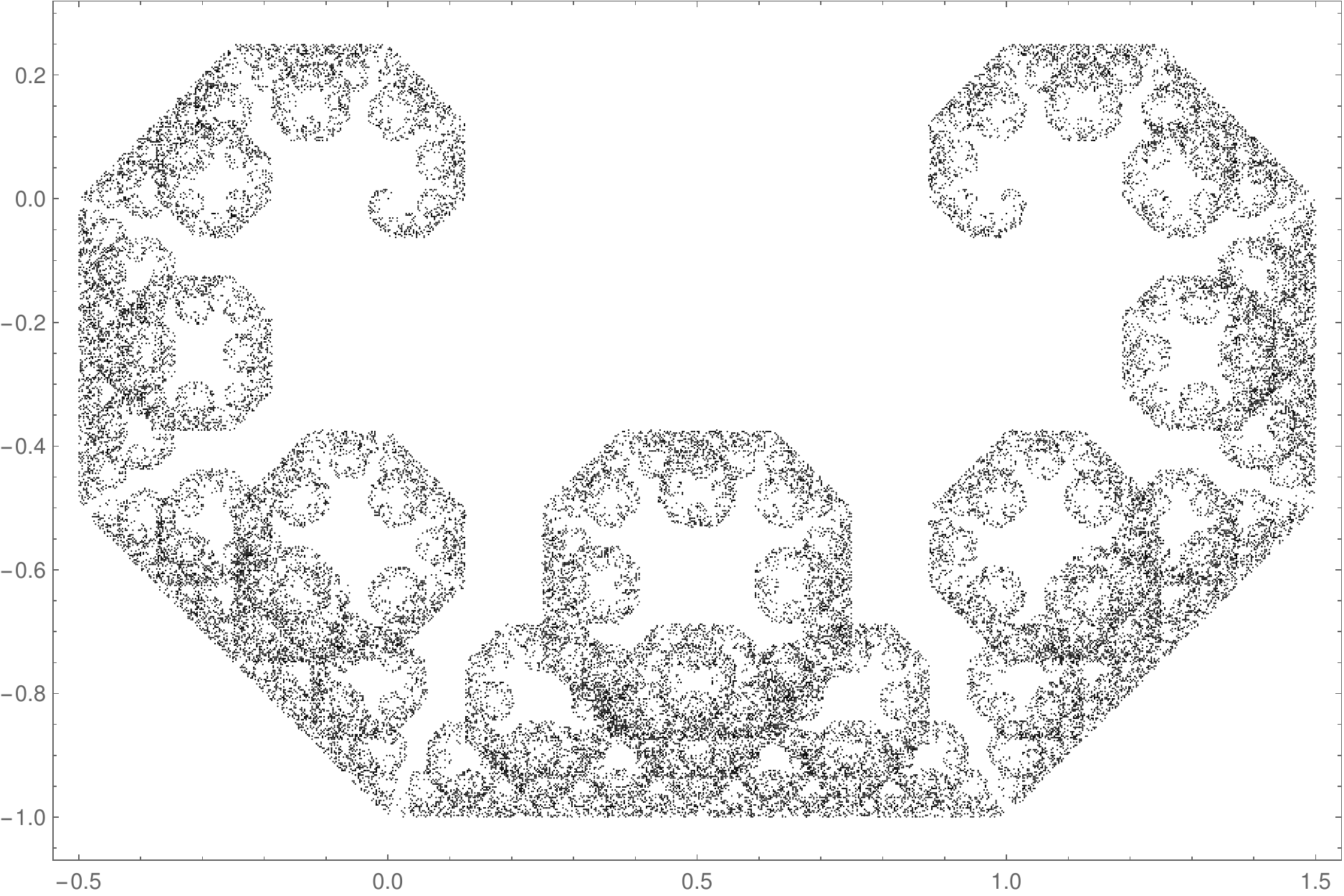}
\caption{For $T = x(x+1)\frac{d^2}{dx^2} + i \frac{d}{dx} + 2$,  the minimal 
Hutchinson-invariant set $M^T_{H,2}$ in degree  $2$ coincides with a L{\'e}vy curve, see \cref{ex:levy}.
}\label{fig:levy}
\end{figure}

In \cref{sec:julia} we abandon the restriction that $T$ is exactly solvable and focus on the case $n=1$. 
In this situation, one has that $\minvset{1}= \minvset{H,1}$.
and without loss of generality, we can then  assume that $T:\bC[x] \to \bC[x]$ acts on polynomials of degree $1$ as 
\[
 T[(x-z)] = U(x) - z V(x),
\]
where $U(x)$ and $V(z)$ are some polynomials. 
Let $\mathcal J_\bC (R)=\mathcal J(R)\cap \bC$, where $\mathcal J(R)$ is the Julia set of the rational function  $R$ of degree $\geq 2$. 
 We  show 
that $\minvset{1}= \minvset{H,1}$ coincides with $\mathcal J_\bC(R)$, where $R(x) = \frac{U(x)}{V(x)}$, under a minor extra assumption on $R$. We call $\mathcal J_\bC(R)$ the \defin{plane Julia set} of $R$.
We also prove that, for larger values of $n$, $\minvset{n}$  contain families of Julia sets, under a small extra assumption---similar results can be found in \cite{AHNST}.

\begin{remark}
Observe that being a Julia set, $\minvset{1}$ looks very different from the ``nice'' convex minimal set $\minvset{\geq 1}$ 
studied in \cite{AlBrSh1}, see examples in \cref{sec:examples}. The same conclusion seems to hold for $\minvset{n}$ in the case $n\geq2$. Computer simulations of  $\minvset{n}$ for $T(p)=(x^2-x+\frac{1}{4})p'+ p$ and several $n$ are given in \cref{fig3}. For $n=1$, this set coincides with the Julia set of $f(x)=x^2+\frac{1}{4}$ (\cref{prop:juliaset}).
When $n\to \infty$, the set $\minvset{n}$ shrinks to $x=\frac{1}{2}$ which is the root of $(x-\frac{1}{2})^2=x^2-x+\frac{1}{4}$, see \cref{th:M_nConv}. Note that for small $n\geq 2$, we do not have a rigorous proof of the existence of $\minvset{n}$.
\end{remark}

\begin{figure}
\centering
\begin{subfigure}[b]{0.2\textwidth}
\includegraphics[width=\textwidth]{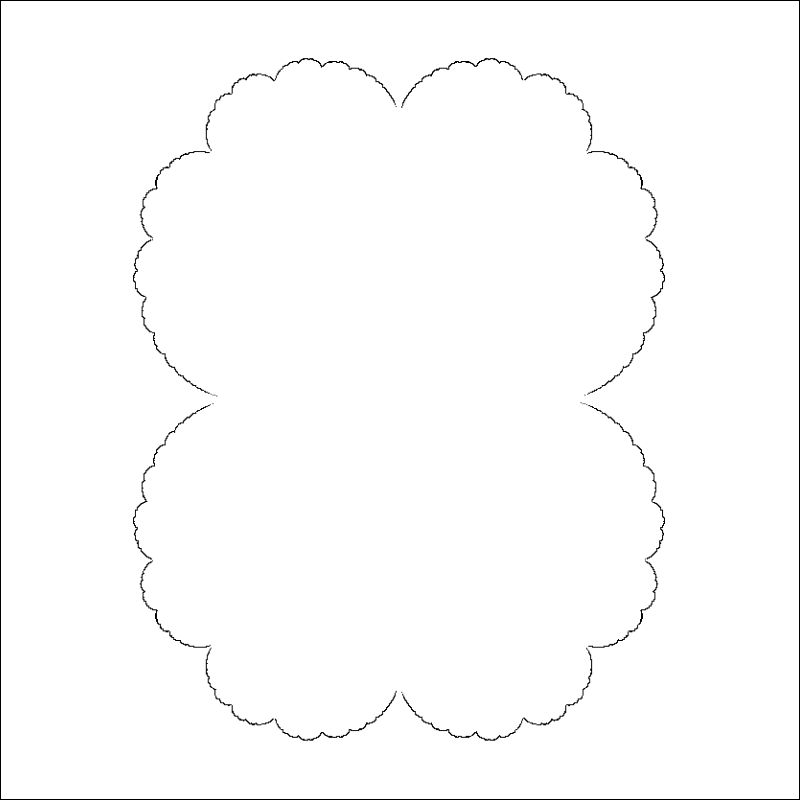}
\caption{$n=1$}\label{fig:julia_c025n1}
\end{subfigure}%
\begin{subfigure}[b]{0.2\textwidth}
\includegraphics[width=\textwidth]{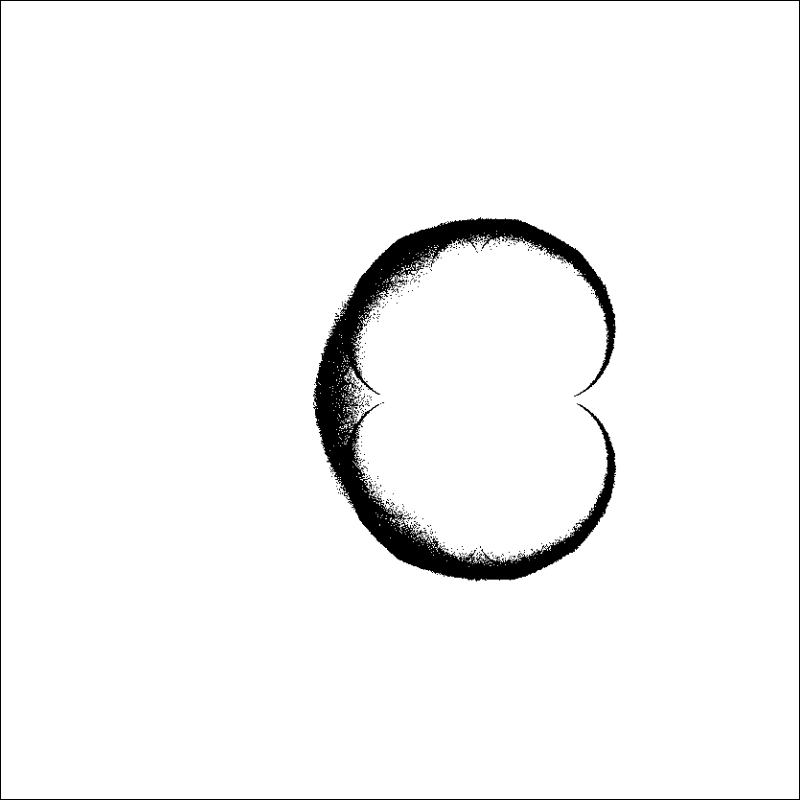}
\caption{$n=2$}\label{fig:julia_c025n2}
\end{subfigure}%
\begin{subfigure}[b]{0.2\textwidth}
\includegraphics[width=\textwidth]{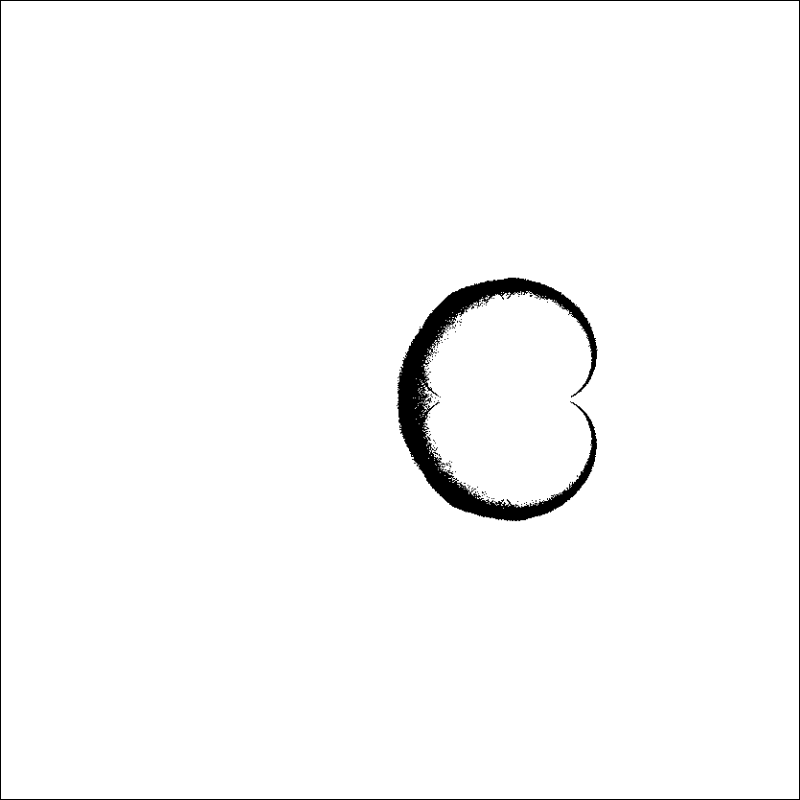}
\caption{$n=3$}\label{fig:julia_c025n3}
\end{subfigure}%
\begin{subfigure}[b]{0.2\textwidth}
\includegraphics[width=\textwidth]{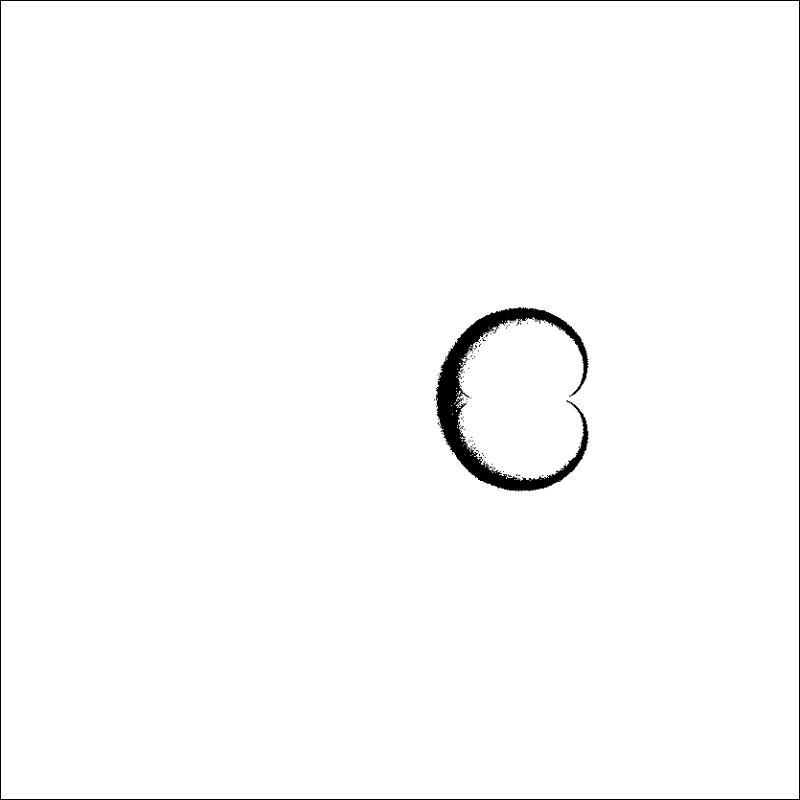}
\caption{$n=4$}\label{fig:julia_c025n4}
\end{subfigure}%
\begin{subfigure}[b]{0.2\textwidth}
\includegraphics[width=\textwidth]{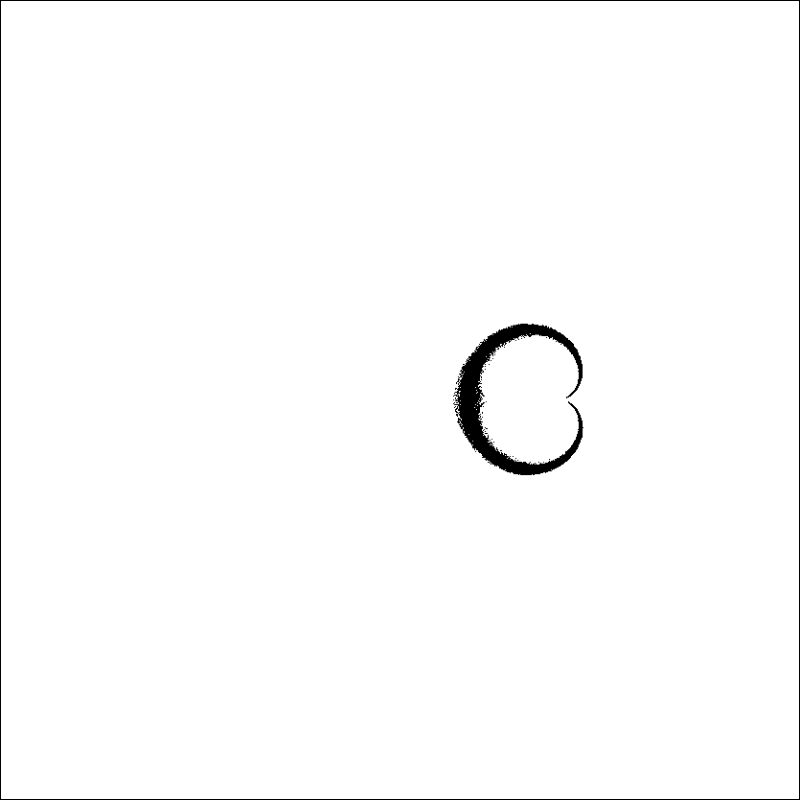}
\caption{$n=5$}\label{fig:julia_c025n5}
\end{subfigure}%
\caption{Computer simulations of the set $\minvset{n}$ for $T(p)=(x^2-x+\frac{1}{4})p'+ p$.
}\label{fig3}
\end{figure}

The structure of the paper is as follows. In \cref{sec:prel} we provide some algebraic preliminaries required for our study. In \cref{sec:general}, 
we present and prove some general results about $\invset{n}$ and $\minvset{n}$. In \cref{sec:finite}, finite $T_n$-invariant sets are studied and we give conditions on $T$ for their existence. \cref{sec:hutchinson} discusses 
Hutchnson-invariant sets and, in particular, $\minvset{H,n}$. 
 \cref{sec:julia} contains our results about the  relation of $\minvset{1}$ for the first order operators, the 
 Julia sets of rational functions, and iterations of algebraic curves. \cref{sec:examples} presents a number of illustrating examples. 
 Finally, \cref{sec:final} contains a number of open problems related to this topic. Needless to say, our paper and its prequel \cite{AlBrSh1} only scratched the surface of the inverse problem in P\'olya-Schur theory.

 \medskip
\noindent
\textbf{Acknowledgements.} 
We are sincerely grateful to Petter Brändén for our discussions and his interest and valuable suggestions.
Research of the third author was supported by the grant  VR 2021-04900 of the Swedish Research Council. He wants to thank Beijing Institute for Mathematical Sciences and Applications (BIMSA) for the hospitality in Fall 2023. 

\section{Algebraic preliminaries}\label{sec:prel}

The next simple statement can be found in \cite[Thm. 1]{BergkvistRullgard2002}. 
 
\begin{THEO}\label{lm:trivial}
The following statements hold. 
\begin{enumerate}[{\normalfont(i)}] \item For any exactly solvable operator $T$ and each non-negative integer $i$, one has 
 \begin{equation}\label{eq:spectrum}
 T(x^i)=\la_i^T x^i+\text{ lower order terms}.
\end{equation}
 Additionally, for $i$ large, the numbers $\la_i^T$ have monotone increasing absolute values. 

\smallskip
\noindent
\item\label{it:thmAit2} For any exactly solvable operator $T$ and any sufficiently large positive integer $n$,
there exists a unique monic eigenpolynomial $p_n^T(x)$ of $T$ of degree $n$.
The eigenvalue of $p_n^T$ equals $\la_n^T$, where $\la_n^T$ is given by \eqref{eq:spectrum}.
\end{enumerate}
\end{THEO}


\begin{remark}
In addition to \cref{lm:trivial}, observe that, for any exactly solvable 
operator $T$ as in \eqref{eq:main} and any non-negative integer $n$, 
$T$ has a (not necessarily unique) monic eigenpolynomial of degree $n$.  \cref{eq:spectrum} implies that 
$T$ is triangular in the 
monomial basis $\{1,x,\dots, x^n,\dots\}$.  In other words, even if $T$ has  multiple 
eigenvalues it has no Jordan blocks, but the eigenpolynomials in the respective degrees are no longer unique. 
A simple example of such situation occurs for $T=x^k\frac{d^k}{dx^k},\; k\ge 2$.
\end{remark}

In what follows, we will use the following two characterizations of exactly solvable operators. 

\begin{proposition} 
Given an arbitrary sequence $\{1, p_1(x), p_2(x), \dots, p_k(x)\}$ 
of monic polynomials  with $\deg p_j=j$ and an arbitrary 
sequence $\{\la_0,\la_1,\dots, \la_k\}$ of complex numbers, 
there exists a unique exactly solvable operator $T$ of order at most $k$ 
with eigenpolynomials $\{1, p_1(x), p_2(x), \dots, p_k(x)\}$ and the 
corresponding eigenvalues $\{\la_0,\la_1,\dotsc, \la_k\}$. 
\end{proposition}

\begin{proof} Consider $T=\sum_{j=0}^k Q_j(x)\frac{d^j}{dx^j}$ with 
undetermined polynomial coefficients $Q_j(x),$ $\deg Q_j\le j$. 
We want to find $T$ such that $T(p_\ell(x))=\la_\ell p_\ell(x)$ for $\ell=0,\dots, k$. 
Let us  find the coefficients $Q_0(x), Q_1(x), \dotsc, Q_k(x)$ from the latter relations. 
 
To start with, observe that since $p_0(x)=1$, we get $T(1)=Q_0(x)=\la_0$ which gives $Q_0(x)$ 
and determines how the operator  $T$ acts on constants. 
Next $T(p_1(x))=\la_1 p_1(x)$ is equivalent to $Q_1(x)+Q_0(x)p_1(x)=\la_1 p_1(x)$. This relation  
can be solved for $Q_1(x)$ and it determines how $T$ acts on linear polynomials. 
Then $T(p_2(x))=Q_2(x)p_2^{\prime\prime}(x)+Q_1(x)p_2^\prime(x)+Q_0(x)p_2(x)=\la_2p_2(x)$. 
Since $p_2^{\prime\prime}(x)=2$, we can determine $Q_2(x)$. 
In general, knowing $Q_0(x), Q_1(x),\dotsc, Q_{\ell-1}(x)$ observe that 
\[
T(p_\ell(x))=\ell! Q_\ell(x)+(\ell)_{\ell-1}Q_{\ell-1}(x)p^{(\ell-1)}(x)+\dots + Q_0(x)p_{\ell}(x)=\la_\ell p_\ell(x)
\]
which allows us to recover $Q_\ell(x)$. (Here and throughout, $(n)_m=n!/(n-m)!$
By construction, $T$ is unique and $\deg Q_\ell\le \ell$. If at least one $\la_\ell \neq 0$, then $T$ is exactly solvable. 
Indeed  if, for all $j=0,\dots, k$,  $\deg Q_j(x)<j$ then $T$ strictly decreases 
the degree of every polynomial it acts upon which gives a contradiction. 
Finally, if all $\la_\ell$ vanish the operator $T$ is identically $0$.
\end{proof}

\begin{proposition} \label{prop:Psi} (i) An operator $T=\sum_{j=0}^kQ_j(x)\frac{d^j}{dx^j}$ of order $k$ with polynomial coefficients $Q_j(x),\; j=0,\dots, k$ is exactly solvable if and only if 
the bivariate polynomial 
\[
\Psi(x,z) \coloneqq T[(x-z)^k]
\]
 has degree  $k$.
 
 \noindent
 (ii) conversely, for any bivariate polynomial $\Psi(x,z)$ of degree $k$, there exists a unique exactly solvable operator $T$ of order {{\bf at most}} $k$ such that $T[(x-z)^k]=\Psi(x,z)$.  
\end{proposition}


\begin{proof} To settle (i), consider 
\[
T[(x-z)^k]=\sum_{j=0}^k(k)_j Q_j(x)(x-z)^{k-j}=
\sum_{j=0}^k(k)_j Q_j(x)\tau^{k-j}\coloneqq \kappa(x,\tau),
\]
where $\tau=x-z$ and suppose that $\Psi(x,z)$ has degree $k$. Obviously, $\Psi(x,z)$ has degree $k$ if and only if $\kappa(x,\tau)$ has degree $k$. Therefore each $Q_j(x)$ has degree at most $j$ and there exists at least one $Q_{j_0}(x)$ whose degree equals $j_0$. Thus $T$ is exactly solvable.  Conversely, suppose that $T$ is exactly solvable, Let $j_0$ be maximal such that $Q_{j_0}(x)$ has degree $j_0$ and let its leading coefficient be $q_{j_0}$. Then $(k)_{j_0}q_{j_0}x^{j_0}(-z)^{k-j_0}$ is the unique term in $\Psi(x,z)$ of degree $j_0$ in $x$ and degree $k-j_0$ in $z$ and there is no term of degree higher than $k$, since this would contradict exact solvability. It follows that $\Psi(x,z)$ has degree $k$. 

\smallskip
To settle (ii), given a bivariate polynomial 
$\Psi(x,z)$ of degree $k$, make a substitution $z\coloneqq x-\tau$ and 
expand $\kappa(x,\tau):=\Psi(x, x-\tau)=\sum_{j=0}^\ell W_j(x)\tau^{k-j}$, where $\ell\le k$. Defining the  coefficients of the operator $T$ as $Q_j(x):=W_j(x)/(k)_j,\; j=0, \dots, \ell$ we obtain an exactly solvable operator of order $\ell\le k$.  
\end{proof}
Notice that the image of the map  sending exactly solvable operators $T$ of order $k$ to $T[(x-z)^k]$ is not surjective on bivariate polynomials of degree $k$. For example, $\Psi(x,z)=(x-z)^k$ is not the image of any operator of order $k$. On the other hand, given  a bivariate polynomial $\Psi(x,z)$ of some degree $k$ and using the above construction,  we always obtain a unique exactly solvable operator of order at most $k$. We denote by $\Phi_k$ the map that takes bivariate polynomials of degree $k$ to an exactly solvable operator of order at most $k$. By construction, it is linear and hence defined on polynomials of degree less than or equal to $k$. Its properties are given in the following proposition.

\begin{proposition}\label{prop:parleur} 
The linear  map $\Phi_k$ has the following properties. 

\begin{enumerate}[{\normalfont(i)}]

\item If $m\leq k$,  then $\Phi_k\left((k)_m Q(x)(x-z)^{k-m} \right)=    Q(x)\frac{d^m}{dx^m}$;

\item if $\deg Q =k$, then
$
\Phi_k\left(Q(x)  \right)= \frac{Q(x)}{k!}\frac{d^k}{dx^k} ;
$

\item if $\deg Q \leq k-\ell$,  then
$
\Phi_k\left( Q(x)z^\ell  \right)= \frac{Q(x)}{k!} \sum_{j=0}^\ell 
 (-1)^j (\ell)_jx^{\ell-j}\frac{d^{k-j}}{dx^{k-j}};$

\item if $\deg Q \le k $, then
$
\Phi_k\left(Q(z) \right)= \frac{1}{k!}\sum_{j=0}^k (-1)^j Q^{(j)}(x) \frac{d^{k-j}}{dx^{k-j}} ;
$

\item if $\deg Q \le k $, then 
$
\Phi_k\left(Q(z)x^\ell  \right)= \frac{x^\ell}{k!}\sum_{j=0}^k (-1)^j Q^{(j)}(x) \frac{d^{k-j}}{dx^{k-j}};
$

\item if $\deg Q\le m$, then 
$
\Phi_k\left(Q(z)(x-z)^{k-m}  \right) \\= \frac{1}{(k)_m}  \sum_{j=0}^m  (-1)^{m-j} \binom{k-j}{k-m} Q^{(m-j)}_m(x)\frac{d^j}{dx^j} .
$
\end{enumerate}
\end{proposition}

\begin{proof} Let us concentrate on  (iii), (iv), (v), and (vi) since (i) and (ii) are obvious. 
Furthermore observe that  multiplication of  $\Psi(x,z)$ by some polynomial $Q(x)$ 
is equivalent to the multiplication of the corresponding 
operator  by the same $Q(x)$.   
Therefore in order to settle (iii),(iv), and (v), it suffices to settle (iv). 
To prove (iv),  consider $\tilde{T}[(x-z)^k]$ where $\tilde{T}$ is the 
differential operator in the right-hand side of (iv).  
One obtains
\[
\tilde{T}[(x-z)^k] = \frac{1}{k!}\sum_{j=0}^k (-1)^j Q^{(j)}(x) (k)_j (x-z)^{k-j}= \sum_{j=0}^k  Q^{(j)}(x) \frac{ (z-x)^{k-j} }{j!}=Q(z).
\]
The latter equality holds because its left-hand side coincides with  the Taylor polynomial 
of degree $k$ for $Q(z)$ taken at the point $x$. 
Since $Q(z)$ itself is a polynomial of degree at most $k$, we get the equality.

\medskip 
By linearity it suffices to prove (vi) only in  case  $Q(x)=x^\ell$ where $0\le \ell\leq m$.
We need to show that
\begin{align*}
 z^\ell(x-z)^{k-m} &= \frac{(k-m)!}{k!} \sum_{j=0}^m (-1)^{m-j} \frac{(k-j)!}{(k-m)!(m-j)!}  Q^{(m-j)}_m(x) \frac{d^j}{dx^j} (x-z)^k \\
                &= \frac{1}{k!} \sum_{j=0}^m (-1)^{m-j} \frac{(k-j)!}{(m-j)!}  \frac{\ell!}{(\ell+j-m)!} x^{\ell+j-m} \frac{k!}{(k-j)!} (x-z)^{k-j} \\
                &= \sum_{j=0}^m (-1)^{m-j} \frac{1}{(m-j)!}  \frac{\ell!}{(\ell+j-m)!} x^{\ell+j-m}  (x-z)^{k-j}. \\
\end{align*}
This can be simplified to showing that
\begin{align*}
 z^\ell &= \sum_{j=0}^m \binom{\ell}{m-j} x^{\ell-(m-j)} (z-x)^{m-j}.
\end{align*}
The latter identity easily  follows from the binomial theorem. 
\end{proof}

%
%
%
%
%
%
%
%

\section{Properties of the minimal \texorpdfstring{$T_n$}{Tn}-invariant set \texorpdfstring{$\minvset{n}$}{In}}\label{sec:general}

\subsection{Existence of \texorpdfstring{$\minvset{n}$}{In}}

\begin{theorem}\label{th:generalN} 
For an exactly solvable operator $T$ and a non-negative integer $n$ such that 

\begin{enumerate}[{\normalfont(i)}]
\item\label{it:1}
among the numbers $\Lambda_n^T=\{\la_0^T, \la_1^T, \dots, \la_n^T\}$ 
there exists a unique $\la_\ell^T,\, \ell\ge 1$ with the largest 
absolute value; 

\item\label{it:2}
there are no $1$-point $T_n$-invariant sets, i.e.  $T_n$-invariant sets being a single point, 
\end{enumerate}
then  there exists  unique   minimal under inclusion nonempty closed set $\minvset{n}\in \invset{n}$. 
 It coincides with  the intersection of all elements of  $\invset{n}$. 
\end{theorem} 

\begin{remark}\label{rk:largen} The assumption~(\ref{it:1}) of Theorem~\ref{th:generalN} is satisfied for an arbitrary exactly solvable operator $T$ if $n$ is sufficiently large in which case, $\ell=n$. 
\end{remark} \begin{proof}
Take any nonempty  $T_n$-invariant set  $S \in \invset{n}$. (Such a set $S$ obviously exists since $\invset{n}$ is never empty; for example, it contains $\bC$.)  By our assumption, $S$ contains at least two distinct points in $\bC$ which we denote by 
 $u\neq v$. Consider the   $(n+1)$-tuple of monic polynomials of degree $n$ given by 
 \begin{equation}\label{eq:basisUV}
   q_j(x) \coloneqq (x-u)^j(x-v)^{n-j}, \; j=0,1,\dots, n.
 \end{equation}
Observe that the collection $\{q_0(x), \dots, q_n(x)\}$  spans $\bC_n[x]$ and that each $T(q_j)$ has all roots in $S$.  Additionally,  take the basis $\{p_0^T, p_1^T, \dots, p_n^T\},$  $\deg p_i^T=i$ of monic eigenpolynomials of $T$ in $\bC_n[x]$ which exists by Theorem~\ref{lm:trivial} (ii). 
(When $T$ has coinciding eigenvalues among $\la_0^T, \la_1^T,\dots, \la_n^T$ such a basis is not unique).
   
Since  $\{q_0(x), \dots, q_n(x)\}$ is a basis in $\bC_n[x]$ there exists $0\le \kappa\le n$ such that the expansion of  $q(x)\coloneqq q_\kappa(x) = \sum_{j=0}^n a_j p_j^T(x)$ has  $a_\ell \neq 0$ where $\ell$ is the index corresponding to the eigenvalue with the largest modulus. 
 Repeated application of $T$ to $q$  is given by
\begin{align}\label{eqn:roots}
T^{\circ m}(q) &= \sum_{j=0}^n a_j \left(\lambda_j^T\right)^m p_j^T(x) 
= \left(\lambda_\ell^T\right)^m \sum_{j=0}^n a_j \left(\frac{\lambda_j^T}{\lambda_\ell^T} \right)^m p_j^T(x).
\end{align} 
Since $S$ is $T_n$-invariant and  $T$ is exactly solvable the roots of $T^{\circ m}(q)$ belong to $S$ for any positive integer $m$. By our assumptions,  $|\lambda_\ell^T|>|\lambda_j^T|,$ for all  $j\neq \ell$ which implies 
that after division by  $\lambda_\ell^m$ the right-hand side of  \eqref{eqn:roots} equals $a_\ell  p_\ell^T(x)$
plus some polynomial of degree at most $n$ all coefficients of which tend to 0 as $m \to \infty$.
Since $S \in \invset{n}$ and $a_\ell\neq 0$, the roots of $p_\ell(x)$ must necessarily belong to $S$.
Hence, the intersection of all  $T_n$-invariant subsets (i.e. all elements of  $\invset{n}$)  must at least contain all roots of $p_\ell(x)$ and therefore  is non-empty since $\ell\ge 1$ by assumption and therefore $\deg p_\ell=\ell>0$. 
This intersection gives the required unique minimal $T_n$-invariant set $\minvset{n}$.
\end{proof}

\begin{remark} If in the 
formulation of \cref{th:generalN}, we keep the assumption (i) and drop the assumption (ii), then the same proof provides the existence of 
the minimal set coinciding with  the intersection of all closed invariant sets 
in $\invset{n}$ each of which contains at least $2$ points. 
The only difference with the case covered 
by \cref{th:generalN} is that in the latter situation, this set will not 
be the unique minimal set contained in $\invset{n}$, since every $1$-point invariant set  is obviously minimal. 
We will study this situation in more detail in \cref{subsec:1-point} below.
\end{remark}

\begin{remark}{Existence of $\minvset{n}$ when several eigenvalues have coinciding maximal absolute value}, in general, fails. As a trivial example of such situation one can consider the operator $T$ which acts on $\bC_n[x]$ by multiplication by a scalar. We will postpone the general question of existence of $\minvset{n}$ to a future paper.
\end{remark}

\subsection{Boundedness and the limiting properties of  \texorpdfstring{$\minvset{n}$}{In}}

\begin{proposition}\label{prop:exactlyfinite}
Let $T$ be an exactly solvable operator with $\lambda_n\neq 0$ and $S$ be an $T_n$-invariant such that $|S|\geq 2$. In this situation,  
\begin{enumerate}[{\normalfont(i)}]
\item\label{it:it1} if there exists $\ell<n$ such that $|\lambda_\ell|>|\lambda_j|,\; j=0,1,\dots, n,\; j\neq \ell$, then $S$ is unbounded, hence infinite or;
\item\label{it:it2}  if $\lambda_\ell\neq 0$ for some $\ell<n$  and there is no integer $m>0$ such that $\lambda_\ell^m=\lambda_n^m$, then $S$ is infinite.
\end{enumerate}
\end{proposition}
\begin{proof}
To settle (\ref{it:it1}), take $u,v\in S$ and let $(p_j^T)$ be the basis of $\mathbb C_n[x]$ consisting of monic eigenpolynomials of $T$ with eigenvalues $\lambda_j$ where (as before) $\deg p_j^T=j$. Let $q_j$ be given by \cref{eq:basisUV}. As $\{q_0,...,q_n\}$ is another basis of  $\mathbb C_n[x]$, there is $q_\kappa(x)=\sum_{j=0}^na_jp_j^T(x)$ with $a_\ell\neq 0$. 
Clearly, $a_n=1$ since $q_\kappa$ is a monic polynomial of degree $n$ by construction. Repeated application of $T$ yields
\[\begin{split}
T^{\circ m}(q_\kappa) &= \sum_{j=0}^n a_j \lambda_j^m p_j^T(x) 
\\&= \lambda_\ell^m \sum_{j=0}^n a_j \left(\frac{\lambda_j}{\lambda_\ell} \right)^m p_j^T(x)
\\&=\ a_\ell p^T\lambda_\ell^m+ \lambda_\ell^m \sum_{0\leq j\leq n, j\neq \ell} a_j \left(\frac{\lambda_j}{\lambda_\ell} \right)^m p_j^T(x).\end{split}\]
Division with $\lambda_\ell^m$ gives 
\[\frac{T^{\circ m}(q_\kappa)}{\lambda_\ell^m}=\ a_\ell p^T_\ell+ \sum_{0\leq j\leq n, j\neq \ell} a_j \left(\frac{\lambda_j}{\lambda_\ell} \right)^m p_j^T(x). \]
The first term in the right-hand side has degree $\ell<n$ and  coefficients independent of $m$.  The second term has degree $n$ and all coefficients of $x^j$  tending to zero for $\ell< j\leq n$. Hence, $(n-\ell)$ zeros of $T^{\circ m}(q_\kappa)$ tend to  $\infty$ as $m\to \infty$ and (\ref{it:it1}) follows.

To settle (\ref{it:it2}) we suppose that there is no $j$ such that $|\lambda_j|>|\lambda_n|$, since then the statement follows from (\ref{it:it1}).
Take $u,v\in S$ and the basis $(p_j^T)$ of $\mathbb C_n[x]$ of monic eigenpolynomials of $T$ with eigenvalues $\lambda_j$ and of degree $j$. Let $q_j$ be given as in \cref{eq:basisUV}. Once again there exists  $q_\kappa$ with $a_\ell\neq 0$, where
\[q_\kappa=\sum_{j=0}^na_jp_j(x)\] 
and $a_n=1$. Application of $T^m$ and division by $\lambda_n^m$ results in 
\[
\frac{T^{\circ m}(q_\kappa)}{\lambda_n^m} =p_n^T +\sum_{j=0}^{n-1} a_j \left(\frac{\lambda_j}{\lambda_n} \right)^m p_j^T(x)
.\] By assumption on $\lambda_j$, and the fact that $\deg (p_j)=j$, there are infinitely many distinct monic polynomials of degree $n$ in the family
\[\left(\frac{T^{\circ m}(q_\kappa)}{\lambda_n^m} \right)_{m=1}^\infty.\]
Hence, this family has infinitely many distinct zeros by the fundamental theorem of algebra. Thus, so does the family $(T^{\circ m}(q_\kappa))_{m=1}^\infty$. 
\end{proof}

\begin{remark}
Observe that even in the case when if $|\la_n|$ has the maximal absolute value  in $\Lambda_n^T$ we cannot 
 conclude  the boundedness of $\minvset{n}$.
In \cref{ex:hutchinson} below, we construct an exactly solvable operator of degree $2$,
such that $|\lambda_0|<|\lambda_1|<|\lambda_2|$ for which  $\minvset{2}$ is unbounded.
\end{remark}


\medskip
The next asymptotic result is an analog of \cite[Theorem 3.14]{AlBrSh1}, see \cref{fig1,fig3}. In this statement, the Hausdorff limits are assumed to be closed. (For an introduction to the Hausdorff metric, see e.g. \cite{Deza2013}.)

\medskip
\begin{theorem}\label{th:M_nConv} 
For any non-degenerate $T$ where at least one zero of $Q_k$ is simple, there is an integer $n_0>0$ such that for each $n\geq n_0$, $\minvset{n}$ exists and the sequence  $\{\minvset{n}\}_{n=n_0}^\infty$ converges in the Hausdorff metric to $Conv(Q_k)$ as $n\to \infty$.
\end{theorem} 
\begin{proof}
First, by \cite[Proposition 4.5]{Hemmingsson2024} for sufficiently large $n$, say for $n\geq n_0$, $\minvset{H,n}$ exist and the sequence $\{\minvset{H,n}\}_{n=n_0}^\infty$ converges in the Hausdorff-metric to the zeros $\mathcal Z(Q_k)$ of $Q_k$, as $n\to \infty$.
Recall that $\minvset{H,n}\subseteq \minvset{n}$ so the latter set exists for sufficiently large $n$.


Since $\lim_{n\to \infty}\minvset{\geq n}=Conv(Q_k)$ \cite[Theorem 3.14]{AlBrSh1}, and $\minvset{n}\subseteq \minvset{\geq n}$ for each $n$, it follows that for each $\epsilon>0$ there is $N$ such that $\minvset{n}$ belongs to the $\epsilon$-neighborhood of $ Conv(Q_k)$. Hence, the statement of the theorem will follow if we can show the existence and convexity of the set $\lim_{n\to \infty}\minvset{n}$. To that end, we follow the ideas of the proof of \cite[ Theorem 2.2 (1)]{AlBrSh1}.

The case when $T$ has order 1 is trivial, so suppose that $T$ has order at least 2. 
Take $z_0,z_1\in \mathcal Z(Q_k)$, $z_0\neq z_1$.  We first show that there is $y\in B(z_0+z_1)/2,\epsilon)\cap \minvset{n}$ for large enough $n$. For each sufficiently large $n$, $ B(z_0,\epsilon/2)\cap \minvset{n}$ and $ B(z_1,\epsilon/2)\cap \minvset{n}$ are non-empty. Note that $\frac{B(z_0+\epsilon/2)+B(z_1,\epsilon/2)}{2}=B(y,\epsilon/2)$. Take now any $x_0\in B(z_0,\epsilon/2),x_1\in B(z_1,\epsilon/2)$, and let us analyze 
\[T\left[(x-x_0)^{\lfloor n/2\rfloor }(x-x_1)^{   \lceil n/2 \rceil }\right]=0\] for sufficiently large $n$. To simplify notation, take $p(x)\coloneqq(x-x_0)^{\lfloor n/2\rfloor }(x-x_1)^{   \lceil n/2 \rceil }$. 
We have 
\[p^{(\ell)}(x)=
 \sum_{j=0}^\ell \binom{\ell}{j}\frac{{   \lfloor n/2 \rfloor !}}{({   \lfloor n/2 \rfloor }-j)!} \frac{{   \lceil n/2 \rceil !}}{({   \lceil n/2 \rceil }+j-\ell)!}(x-x_1)^{\lfloor n/2 \rfloor  - j} (x-x_2)^{\lceil n/2 \rceil  +j-\ell}.
\]
Thus, expanding the rational expressions in $n$ we get 
\[\begin{split}
\frac{p^{(\ell)}(x)}{n^{\ell}(x-x_0)^{\lfloor n/2\rfloor-\ell }(x-x_1)^{   \lceil n/2 \rceil -\ell}} &= \sum_{j=0}^\ell \binom{\ell}{j} (x-x_1)^{\ell-j} (x-x_2)^{j} +O(n^{-1})R_\ell(x)
\\&=  \left((x-x_1)+(x-x_2)\right)^\ell + O(n^{-1}) R_{\ell}(x),\end{split} 
\]
as $n\to \infty$, for some polynomial depending on $\ell$ but not on $n$. Just as in \cite[ Theorem 2.2 (1)]{AlBrSh1}, for large $n$ there are  zeros of $T(p(x))$ close to the zeros of $ \left((x-x_1)+(x-x_2)\right)$. In particular, for sufficiently large $n$ there are zeros of $T(p(x)) $ which are less than $\epsilon/2$-close to $\frac{x_1+x_2}{2}$ and therefore less than $\epsilon$-close to $y$. 

As the above argument holds for all $x_0\in  B(z_0,\epsilon/2)$ and $x_1\in B(z_1,\epsilon/2)$ and for sufficiently large $n$, $ B(z_0,\epsilon/2)\cap \minvset{n}$ and $ B(z_1,\epsilon/2)\cap \minvset{n}$ are non-empty. Thus for sufficiently large $n$, $\minvset{n}\cap B(y,\epsilon)\neq \emptyset$ for sufficiently large $n$. One can for each $z\in Conv(Q_k)$ and each $\epsilon>0$ repeat the argument above argument a finite number of times to obtain that for sufficiently large $n$, there is a point in $B(z,\epsilon)\cap \minvset{n}$. Since $\lim_{n\to \infty}\minvset{\geq n}=Conv(Q_k)$ we obtain the result.
\end{proof}
%
%
%


\subsection{Forward iterations and alternative characterization of \texorpdfstring{$\minvset{n}$}{In}}\label{subsec:forwardIteration}

In this subsection we will describe  connections of the problem of finding $T_n$-invariant sets with complex dynamics. 

\medskip
Given an arbitrary subset $\Omega$ of the complex plane, let $\cP_n(\Omega)$ 
be the set of all polynomials of degree $n$ with all zeros in $\Omega$.

\begin{definition}
For an arbitrary set $\Omega \subseteq \bC$ and an operator $T$ given by \eqref{eq:main}, we define $\tau_n(\Omega)$ as given by 
\begin{equation}\label{eq:forwardIterate}
\tau_n(\Omega) \coloneqq  \Omega \cup \{ x \in \bC ; x \text{ is a root of } T(p(x)), \text{ for some } p \in \cP_n(\Omega) \}. 
\end{equation}

\smallskip
Further we define $G_n(\Omega)$ as 
\[
G_n(\Omega) \coloneqq \bigcup_{j=0}^\infty \tau_n^{\circ j}(\Omega), 
\]
where $\tau_n^{\circ j}$ stands for the $j$-th  iteration of the operation $\tau_n$. 
In other words, $G_n(\Omega)$ is the union of the  iterations of $\tau_n$ applied to $\Omega$.
\end{definition}

Observe that $S \in \invset{n}$ if and only if $\tau_n(S) = S$ and $S$ is closed.

\begin{lemma}
For any operator $T$, positive integer $n$ and any $\Omega\subseteq \bC$, the closure $\overline{G_n(\Omega)}$ is $T_n$-invariant, i.e. belongs to  $\invset{n}$. 
\end{lemma}
\begin{proof} 
Let $S \coloneqq \overline{G_n(\Omega)}$. 
It suffices to show the inclusion $\tau_n(S)\subseteq S$.
 Now, 
\[
\tau_n(S) = \tau_n\left(
\overline{ \bigcup_{j=0}^\infty \tau_n^{\circ j}(\Omega) } 
\right) \subseteq
\overline{ \bigcup_{j=1}^\infty \tau_n^{\circ j}(\Omega) }
\subseteq 
\overline{ \bigcup_{j=0}^\infty \tau_n^{\circ j}(\Omega) }
=S
\]
which proves the assertion.
\end{proof}

By construction, $\overline{G_n(\Omega)}$ is the smallest closed $T_n$-invariant  set containing $\Omega$.
We can now present a similar characterization  of the minimal closed $T_n$-invariant set $\minvset{n}\in \invset{n}$. 
\begin{proposition}\label{claim1}
A set $S \in \invset{n}$ is minimal (under inclusion), i.e. $S=\minvset{n}$ if and only if $\overline{G_n(S_1)} = \overline{G_n(S_2)}$
for any two non-empty subsets $S_1,S_2 \subseteq S$.
\end{proposition}
\begin{proof}
Suppose that there are $S_1,S_2 \subseteq S$ such that $\overline{G_n(S_1)} \neq \overline{G_n(S_2)}$.
Then the two latter sets are two different closed subsets of $S$,
and $S$ cannot be minimal.

In the opposite direction, suppose $S$ is not minimal. Then there exists some closed $S_1 \subseteq S$
which is strictly contained in $S$ and that is  $T_n$-invariant. Then we get that 
$S_1 = \overline{G_n(S_1)}$ and $\overline{G_n(S)} = S$ implying 
$\overline{G_n(S_1)} \neq \overline{G_n(S_2)}$.
\end{proof}

\begin{proposition}
Every infinite, bounded and closed  minimal $T_n$-invariant set $S \in \invset{n}$ is dense in itself.
\end{proposition} 
\begin{proof} 
Let $a$ be an accumulation point in $S$ ---
 such a point exists, since $S$ is bounded and contains infinitely many points.
Since $S$ is minimal, it follows from \cref{claim1} that $S = \overline{G_n(\{a\})}$.
Let $b \in S$ be any point and let $\Omega_{b,\epsilon}$ be an $\epsilon$-neighborhood of $b$.

Using minimality of $S$, it follows that there is some $j$ (depending on $a$, $b$ and $\epsilon$) 
such that the set $\tau_n^{\circ j}(\{a\})$ intersect $\Omega_{b,\epsilon/2}$.
Now, since $a$ was chosen to be an accumulation point, 
there are infinitely many points of $S$ which are in $\Omega_{b,\epsilon}$.
\end{proof}


\medskip
\begin{proposition}\label{pr:M_n} 
Suppopse that $T$ is exactly solvable and there are no one-point $T_n$-invariant sets for sufficiently large $n$. Then for sufficiently large $n$,  $\minvset{n}$ exists and is given by $\overline{G_n(\{z\})}$,
where $z$ is any of the roots of the eigenpolynomial $p_n^T$.
\end{proposition} 
\begin{proof}
From \cref{lm:trivial} and the proof of \cref{th:generalN} we have that whenever $n$ is sufficiently large,
every closed minimal set contains the roots of the eigenpolynomials $p^T_n$ (comp. \cref{th:generalN}~(\ref{it:thmAit2})).
This implies that $\minvset{n} = \overline{G_n(\{z\})}$ where $z$ is 
any root of $p^T_n$.
\end{proof}

\section{(Non)existence of finite \texorpdfstring{$T_n$}{Tn}-invariant sets}\label{sec:finite}

\subsection{\texorpdfstring{$1$}{1}-point \texorpdfstring{$T_n$}{Tn}-invariant sets}\label{subsec:1-point} 
In Theorem~\ref{th:generalN} we used the assumption that no $T_n$-invariant set consists of a single point. 
Below we give necessary and sufficient conditions for the existence of $1$-point $T_n$-invariant sets in $\invset{n}$ and discuss how they affect the situation.  
in other words, we consider one situation not covered by \cref{th:generalN}. 
Namely, we shall understand what happens if $\la_\ell$ is still unique but $\invset{n}$ contains $1$-point $T_n$-invariant sets.  The case when there exist several $\la_j$'s with the same maximal absolute value is not completely understood at the present moment.

\medskip
Given an exactly solvable operator $T$ and a positive integer $n$, set 
\[
\Psi_{T,n}(x,z) \coloneqq T[(x-z)^n],
\]
where $\Psi_{T,n}(x,z)$ is some bivariate polynomial of degree at most $n$ different from a constant.
(If $\Psi_{T,n}(x,z)=\text{const}$ 
then every subset of $\bC$ is an invariant set.) 

Denote by $\Ga_{T,n}\subseteq \bC^2$ the algebraic curve given by $\Psi_{T,n}(x,z)=0$. 
Here, the space $\bC^2$ is equipped with the standard coordinates $(x,z)$.

\begin{theorem}\label{th:1-pt} 
In the above notation, given an exactly solvable operator $T$, the following holds.
\smallskip
\noindent
\begin{enumerate}[{\normalfont(i)}]
\item\label{it:item1} A point $z_0\in \bC$ is a $1$-point $T_n$-invariant set  if and only if: 
  \noindent
  \begin{itemize}\item either the intersection $\Ga_{T,n} \cap L_{z_0}\subseteq \bC^2$ contains only the point $(z_0,z_0)$, where $L_{z_0}$ is the line $\{z=z_0\}$, 
  
  \smallskip
  \noindent
 \item or $\Ga_{T,n} \cap L_{z_0}\subseteq \bC^2=L_{z_0}$  in which case $T[(x-z_0)^n]\equiv 0$.
\smallskip
\noindent
\end{itemize}
\item\label{it:item2}There exist infinitely many distinct $1$-point invariant sets  if and only if $\Psi_{T,n}(x,z)=\Phi(z)(x-z)^\ell$ 
for some $1\le\ell\le n$ and $\deg \Phi(z)\le n-\ell$. (In this case every point $p\in \bC$ is a $1$-point invariant set).
\end{enumerate}

\end{theorem}
\begin{proof} Indeed, to settle (\ref{it:item1}) observe that, by definition, a $1$-point $T_n$-invariant subset  is a point $z_0\in \bC$ such that the polynomial $T[(x-z_0)^n]$  either vanishes identically or has a unique complex root at $z_0$ and no other roots. Again, by definition, for any fixed $\tilde z \in \bC$, the roots of the univariate polynomial $T[(x-\tilde z)^n]$ are given by the intersection $\Ga_{T,n} \cap L_{\tilde z}$.  Thus if $z_0$ is a $1$-point invariant subset and $T[(x-z_0)^n]$ is not identically zero, then $\Ga_{T,n} \cap L_{z_0}$ contains the point $(z_0,z_0)$ and no other points. If, on the other hand, $T[(x-z_0)^n]\equiv 0$  then $\Ga_{T,n} \cap L_{z_0}=L_{z_0}.$ 

\smallskip
To settle (\ref{it:item2}) observe that if $\{z_0\}$ is a $1$-point $T_n$-invariant set, then either the point $(z_0,z_0)$ belongs to $\Ga_{T,n}$ or  $\Ga_{T,n} \cap L_{z_0}=L_{z_0}$. If  $T[(x-z)^n]$ is a non-constant bivariate polynomial, then given an infinite sequence  $\Z=\{z_i\}_{i=1}^\infty$ of pairwise distinct $1$-point $T_n$-invariant sets, we get that only finitely many of them can correspond to full fibers  $\Ga_{T,n} \cap L_{z_i}=L_{z_i}$. Thus an infinite (sub)sequence of  $\Z$ consists of points $z_i$ such that $(z_i,z_i)\in \Ga_{T,n}$. Since $\Ga_{T,n}$ is an algebraic curve, then it must contain the whole diagonal line $\{x=z\}$ with some positive multiplicity. Thus $\Ga_{T,n}$ is the union of some horizontal lines given by finitely many  values of the variable $z$ (which might occur with nontrivial multiplicities) and the diagonal line $\{z=x\}$ with some multiplicity. Since the degree of $\Ga_{T,n}$ is at most $n$ this implies that $\Psi_{T,n}(x,z)=\Phi(z)(x-z)^\ell$ for some $\ell\le n$ and $\deg \Phi(z)\le n-\ell$.
\end{proof} 

\begin{lemma}\label{lm:max} If the number of $1$-point $T_n$-invariant sets is finite, then there are at most $n$ of them, and this bound is sharp. 
\end{lemma}

\begin{proof} $1$-point $T_n$-invariant sets  form a subset of the intersection $\Ga_{T,n}\cap \{z=x\}$. Since the degree of $\Ga_{T,n}$ is at most $n$ and the line $\{z=x\}$ is not a component of $\Ga_{T,n}$(by \cref{th:1-pt} (ii), the upper bound follows.  To show that there exist operators $T$ with $n$ $1$-point $T_n$-invariant sets in $\invset{n}$, consider the sequence $\{Hyp_n\}$ of hyperelliptic algebraic curves given by: $Hyp_n:  \{(x-z)^2=(z-1)(z-2)\dots (z-n)\}.$  Observe that $\deg Hyp_n=n$ and that the intersection of $Hyp_n$ with the diagonal $\{x=z\}$ consists of the $n$-tuple of points $(1,1); (2,2); \dots, (n,n)$. Finally, for $j=1,\dots, n$, the intersection of $Hyp_n$ with the horizontal line $\{z=j\}$ is given by the equation $(x-j)^2=0$ which implies that $(j,j)$ is the only point in this intersection. 
\end{proof}

\begin{proposition} 
\label{prop:families} An exactly solvable operator $T$ given by \eqref{eq:main} has infinitely many $1$-point invariant  sets in $\invset{n}$ if and only if the truncation of $T$ up to derivatives of order $n$ belongs to one of the families $\F_{m,n},\; m=0,1,\dots, n-1$. Here  
\begin{equation}\label{eq:Fmn}
\F_{m,n}= \sum_{j=0}^m  (-1)^{m-j} \binom{n-j}{n-m} Q^{(m-j)}_m(x)\frac{d^j}{dx^j},
\end{equation}
where $Q_m(x)$ is an arbitrary polynomial of degree at most $m$. 
\end{proposition}
\begin{remark}
By truncation of $T$ up to derivatives of order $n$, we mean that all terms of $T$ as in \eqref{eq:main} corresponding to $j\leq n$ are the only ones that are considered.
\end{remark}

\begin{proof} In view of part (ii) of \cref{th:1-pt},  it suffices to check that for $T$ given by \eqref{eq:Fmn}, we have 
\[
T[(x-z)^n]= (n)_m Q_m(z)(x-z)^{n-m}.
\]

Since we are interested in the action of $T$ on $\bC_n[x]$ we can from the beginning truncate $T$ 
up to the derivatives of order $n$ since higher derivatives act trivially on  $\bC_n[x]$.  
After that observe that the right-hand side of \eqref{eq:Fmn}  coincides (up to the factor $(n)_m$ 
and substitution $k=m$) with the formula (vi) of \cref{prop:parleur} 
which corresponds exactly to the polynomial $Q_m(z)(x-z)^{n-m}$. 
\end{proof}

\begin{example} As an illustration of \cref{prop:families}, for $n=1$, we have the family $\mathcal F_{0,1}=a$, where $a\in \bC$. For $n=2$, we  have two families; $\mathcal F_{0,2}=a$, and $\mathcal F_{1,2}=(ax+b)\frac{d}{dx}-2a$, where $a,b\in \bC$.   
For $n=3$, we have three families; $\mathcal F_{0,3}=a$, $\mathcal F_{1,3}=(ax+b)\frac{d}{dx}-3a$,  and $\mathcal F_{2,3}=(ax^2+bx+c)\frac{d^2}{dx^2}-(4ax+2b)\frac{d}{dx}+6a$, where $a,b,c\in \bC$.
\end{example}

\subsection{(Non)existence of finite invariant sets of cardinality exceeding \texorpdfstring{$1$}{1}} 
%
%
%

The following result has been proven in \cite{Hemmingsson2024}.
\begin{lemma}[{\cite[Lemma 4.3]{Hemmingsson2024}}]\label{le:perfect}
Suppose that $T$ is non-degenerate opeator of order $k\ge 2$ and that there is a simple 
zero $u_1$ of $Q_k(x)$ such that  not all $Q_j(x)$ have $u_1$ as a common zero. Then $ \minvset{H,n}$ is a perfect set for $n$ sufficiently large.
\end{lemma}
Recall that $ \minvset{H,n}$ is the minimal Hutchinson-invariant set in degree $n$, see Definition~\ref{def7}.  Since $ \minvset{H,n}\subseteq \minvset n$, this shows that under the assumptions of \cref{le:perfect}, $ \minvset{n}$ is infinite.

\begin{proposition}\label{prop:finiteInvertibleCase}  
Let $T: \bC_n[x]  \to \bC_n[x]$ be an invertible linear operator that sends all degree $n$ polynomials to degree $n$ polynomials. Suppose further that $S\subseteq \bC$ is a finite $T_n$-invariant set of cardinality at least $2$.
Let $\alpha$ be such that $T[x^n]=\alpha x^n +$lower order terms.    Then there exists $k \in \bN$ such that $T^k[p]=\alpha^k p$ for all polynomials $p\in \bC_n[x]$.   
\end{proposition}

\begin{remark}The  condition that $T$ sends polynomials of degree $n$ to polynomials of degree $n$  implies that $T$ preserves the subspace $\bC_{n-1}[x]\subseteq \bC_n[x]$.\end{remark}
\begin{proof} 
Assume that $T : \bC_n[x]  \to \bC_n[x]$ is an invertible linear operator 
and that $S= \{z_1,\dotsc,z_r\}$ is a finite $T_n$-invariant set  with  $r\ge 2$ points.
Let $P = \{p_1,\dotsc,p_\ell\}$ be the \emph{finite} set of all 
monic polynomials of degree $n$ with roots in $S$.
Define the operator $\Tilde T$ by $\Tilde T[p]=T[p]/\alpha.$ Note that any set is $T_n$-invariant if and only if it is $\Tilde T_n$-invariant. Since $\Tilde T$ is invertible then for each $j\in \{1,\dots, \ell\}$, there exists $i_j\in \{1,\dotsc, \ell\}$ 
such that  $\Tilde T[p_j] = p_{i_j}$
Indeed, we cannot have that 
\[
 \Tilde T[p_j] =\Tilde T[p_i] \text{ for any } i \neq j,
\]
since then $\Tilde T[p_j - p_i ] = 0$, but $ p_j - p_i$ is non-zero
and this would mean that $\Tilde T$ and hence $T$ is non-invertible.
It follows that there exists a positive integer $k$ such that every $p \in P$ is an eigenvector of $\Tilde T^k$ with eigenvalue 1,
 where $\Tilde T^k: \bC_n[x]  \to \bC_n[x]$ is the $k$-th power of $\Tilde T$. 
Observe that $\Tilde T^k$ also has $S$ as a $\Tilde T^k_n$-invariant set.

Let $u\neq v$ belong to $S$.
 Observe that the polynomials
 \begin{equation}\label{eq:basis}
   q_j(x) \coloneqq (x-u)^j(x-v)^{n-j}, \; j=0,1,\dots, n,
 \end{equation}
form a basis for $\bC_n[x]$. Additionally, since they belong to $P$ we have that $\Tilde T^k[q_j] = q_j$.
Take any polynomial $p(x)\in \bC_n[x]$ and present it as  

\[
p(x)=\sum_{j=0}^n a_j q_j(x).  
\]
By linearity of $T^k$, we have that
\begin{align}\label{eq:oneSide}
 \Tilde T^k[(p(x)] = \Tilde T^k\left[  \sum_{j=0}^n a_j q_j(x)   \right] = \sum_{j=0}^n a_j q_j(x)=p(x).
\end{align} We 
conclude that $T^k[p]=\alpha^k p$ for all polynomials $p\in \bC_n[x]$. 
\end{proof}

For a positive integer $n$,  introduce $p_{k,n}=(x-1/2)^{n-k}(x+1/2)^{k}$ for $0\leq k \leq n$.
In order to prove the next statement, we will need the following lemma for the upcoming \cref{prop:cardinality 2}. 

\begin{lemma}\label{le:1/2}
For $0\leq m\leq n$,
\begin{equation}\label{eq:-1/2}(-1)^{n-m}(x-1/2)^m=\sum _{k=0}^{n-m} (-1)^k p_{k,n}(x) \binom{n-m}{k},\end{equation}
and
\begin{equation}\label{eq:+1/2}(-1)^{n-m}(x+1/2)^m=\sum _{k=0}^{n-m} (-1)^k p_{k+m,n}(x) \binom{n-m}{k},\end{equation}
where \( p_{k,n}(x) = (x-1/2)^{n-k}(x+1/2)^k \) for \( 0 \leq k \leq n \).
\end{lemma}
\begin{proof}
We show the validity of \cref{eq:-1/2}, the proof for \cref{eq:+1/2} is analogous. We do this by induction on $n$. In the case that $n=0$ or $ n\geq 1 $ and $m=n$, the statement is immediate.

Now  for$n\geq 1$ and $m\leq n $,  assume the validity of the induction hypothesis
\[
(-1)^{n-m}(x-1/2)^m = \sum_{k=0}^{n-m} (-1)^k p_{k,n}(x) \binom{n-m}{k}.
\]
We need  to show that 
\begin{equation*}
(-1)^{(n+1)-m}(x-1/2)^m = \sum_{k=0}^{(n+1)-m} (-1)^k p_{k,n+1}(x) \binom{(n+1)-m}{k}.
\end{equation*}

Using Pascal's identity, we may rewrite the right-hand side as
\[
\sum_{k=0}^{(n+1)-m} (-1)^k p_{k,n+1}(x) \binom{(n+1)-m}{k}= \]\[\begin{split}&= \sum_{k=0}^{(n+1)-m} (-1)^k p_{k,n+1}(x) \binom{n-m}{k} + \sum_{k=0}^{(n+1)-m} (-1)^k p_{k,n+1}(x) \binom{n-m}{k-1}
\\&
=\sum_{k=0}^{n-m} (-1)^k p_{k,n+1}(x) \binom{n-m}{k} - \sum_{j=-1}^{n-m} (-1)^{j+1} p_{j+1,n+1}(x) \binom{n-m}{j}
\\&
=\sum_{k=0}^{n-m} (-1)^k p_{k,n+1}(x) \binom{n-m}{k} - \sum_{j=0}^{n-m} (-1)^{j+1} p_{j+1,n+1}(x) \binom{n-m}{j}
\\&
=\sum_{k=0}^{n-m} (-1)^k \binom{n-m}{k} ( p_{k,n+1}(x) - p_{k+1,n+1}(x) )
\\&
=\sum_{k=0}^{n-m} (-1)^k \binom{n-m}{k} ( (x+1/2)^k(x-1/2)^{n-k}((x-1/2)-(x+1/2))
\\&
=-\sum_{k=0}^{n-m} (-1)^k \binom{n-m}{k} p_{k,n}(x)=-(-1)^n(x-1/2)^m=(-1)^{n+1}(x-1/2)^m,
\end{split}
\]
where we have introduced $j=k-1$ and used the induction assumption in the second to last equality.
This calculation completes the inductive step and proves the lemma.
\end{proof}

\begin{proposition}\label{prop:cardinality 2}  
Let $T: \bC_n[x]  \to \bC_n[x]$ be an invertible linear operator that sends all degree $n$ polynomials to degree $n$ polynomials. Suppose further that $S\subseteq \bC$ is a finite $T_n$-invariant set of cardinality $2$. After a linear change of variables, we have that either $T$ is the identity or $T(p_{k})=p_{n-k,n}$ for each $0\leq k\leq n$. Moreover, $T$ is induced by the change of variaibles $z\to -z$ in the sense that after multiplying $T$ by $(-1)^n$, we have
\[
T\left[\prod_{i=1}^m(z-a_i)\right]=\prod_{i=1}^m(- z-a_i)
\]
provided that $0\leq m\leq n$. In particular, for any $N\in \mathbb N$ there are finite $T_n$-invariant sets of cardinality $N$.
\end{proposition}
\begin{proof}
After a linear change of variables, we can assume that $S=\{-1/2,1/2\}$ and $$T(x^n)=x^n+\text{ lower order terms}.$$ By the proof of \cref{prop:finiteInvertibleCase}, we have that $T$ is a permutation on the set $\{p_{k,n},k=0,\dots,n\}$. Using \cref{le:1/2}, let us write 
\[
(-1)^n=\sum_{k=0}^n(-1)^k\binom nk p_{k,n}(x) 
\]
and evaluate 
\begin{equation*}\label{eq:T1}
T[(-1)^n]=\sum_{k=0}^n(-1)^k\binom nk T[p_{k,n}(x)]. \end{equation*}

As $T$ is exactly solvable, we have that $T[(-1)^n]$ is a constant. 
The set  $\{p_{k,n},k=0,\dots,n\}$ forms a basis of the set of polynomials of degree at most $n$.
Hence, in the case that $n$ is odd, for which $(-1)^{k_0}\binom {n}{k_0}\neq (-1)^{k_1}\binom {n}{k_1}$ for $k_0\neq k_1$, it follows 
that $T(p_{k,n})=p_{k,n}$ for each $k$ or $T(p_{k,n})=p_{n-k,n}$ for each $k$. 
Further, in the case that $n$ is even it follows that for each $k$, either $T(p_{k,n})=p_{k,n}$ or $T(p_{k,n})=p_{n-k,n}$
and if $T(p_{k,n})=p_{n-k,n}$, then $T(p_{n-k,n})=p_{k,n}$. We have thus shown that either $T$ is the identity or $T(p_{k,n})=p_{n-k,n}$ 
in the case that $n$ is odd, so let us now assume that $n$ is even.
We use \cref{le:1/2} to write
\[
(x+1/2)=\sum _{k=0}^{n-1} (-1)^{k+1} p_{k+1,n}(x) \binom{n-1}{k}
\] 
and look at   
\begin{equation}\label{eq:Tx}
T[(x+1/2)]=\sum_{k=0}^{n-1}(-1)^{k+1}\binom {n-1}{k} T[p_{k+1,n}(x)].
\end{equation}
In order for $T$ to be exactly solvable and invertible, we need that the expression above is a polynomial of degree 1.

There are two cases to consider, either $T(p_{n,n})=p_{0,n}$, in which case $T(p_{0,n})=p_{n,n}$ or $T(p_{n,n})=p_{n,n}$. 
In the first case, we have that all summands in the RHS of \cref{eq:Tx} has a zero at $1/2$ and that $\{p_{k},k=0,\dots,n-1\}$ forms a 
basis of the span of the set of polynomials with a zero at $1/2$. By \cref{le:1/2},
\[
(x-1/2)=\sum _{k=0}^{n-1} (-1)^{k+1} p_{k,n}(x) \binom{n-1}{k},
\]
and  we find that $T[p_{k+1}]=p_{n-1-k}$ for each $k$ as desired.
If however $T({p_n})=p_n$, we find that $\{p_{k+1},k=0,\dots,n-1\}$ forms a basis of the span of the set of 
polynomials with a zero at $-1/2$. In the same way as above, we find that $T(p_{k,n})=p_{k,n}$ for each $k$.

Now, if $T(p_{n-k,n})=p_{k,n}$ for each $k$, let $\tilde T=(-1)^n T$. For any polynomial $P$ of degree at most $n$, write 
\[
P=\sum_{k=0}^n a_k (x-1/2)^k.
\]
Then \cref{le:1/2} gives, after some manipulations including change of indices, that 
\[
\tilde T(P)=\sum_{k=0}^n(-1)^ka_k (x+1/2)^k=\sum_{k=0}^na_k (-x-1/2)^k.
\]
This shows that $\tilde T(P(x))=P(-x)$ as asserted. 
\end{proof}

\begin{remark} \label{lm:exist}  If we remove the condition of invertibility then finite $T_n$-invariant sets might exist
for a larger set of operators $T$. Indeed, if $T$ acts on $\bC_n[x]$ with a $1$-dimensional 
image spanned by some polynomial $Q(x)$ of positive degree, then the zero locus of $Q(x)$ is a $T_n$-invariant set. 
 \end{remark}

\section{Hutchinson-invariant sets  and iterations of  plane algebraic curves}\label{sec:hutchinson}

\subsection{Hutchinson set-up}\label{sec:hutsetup}
In this section we describe the connection between the study of Hutchinson-invariant sets, see \cref{def7}, and iterations of   
plane algebraic curves which is an important but insufficiently explored area of complex dynamics.  Such iterations was first suggested by P.~Fatou \cite{Fatou}, but the amount of publications devoted to this topic is rather limited and mainly concentrated around special examples, see e.g. \cite{BhSr2016, Bullett1988, BullettPenrose1994, BullettPenrose2-1994, Bullett2010, BullettLomonacoSiqueira, BullettLomonaco, Gumenyuk, LeeLyubichMakarov}. 

\medskip
Recall that, for a given exactly solvable operator $T$, its Hutchinson-invariant set in degree $n$ is a subset $S\subseteq \bC$ such that, for any $z\in S$, $T(x-z)^n$ has all zeros in $S$ or vanishes identically.  This definition can be interpreted as follows. As above consider the bivariate polynomial $\Psi_{T,n} \coloneqq  T[(x-z)^n]$ and   the algebraic curve $\Ga_{T,n}\subseteq \bC^2$ given by $\Psi_{T,n}(x,z)=0$. (The space $\bC^2$ is equipped with the standard coordinates $(x,z)$.) 
By Proposition~\ref{prop:Psi}, $\deg \Psi_{T,n}(x,z)=n$. 

For any given $z_0$, the roots of $T[(x-z_0)^n]$ are the (projections onto the $x$-axis of the) intersection points in $\Ga_{T,n}\cap L_{z_0}$, where the line $L_{_0}$ is given by $z=z_0$. Thus one can consider the natural multi-valued map $\Theta_{T,n}$ sending every $z_0\in \bC$ to $\Ga_{T,n}\cap L_{z_0}$.  The reason for these considerations is clear; by definition, a closed set $S\subseteq \bC$ is $T_{H,n}$-invariant if and only if it is mapped to itself by $ \Theta_{T,n}$. 
If the (not necessarily distinct) irreducible components of $\Ga_{T,n}$ are $\{\Gamma_i;i=1,\dots,n\}$, we denote by $\theta_{T,i}$ the multivalued maps sending each $z_0\in \bC$ to $\Ga_{i}\cap L_{z_0}$. 
To make the domain and the range spaces compact and well-defined, one usually compactifies $\bC^2$ as $\bC P^1_x\times \bC P^1_z$ 
with $x$ being the affine coordinate on $\bC P^1_x$ and $z$ being the affine coordinate on $\bC P^1_z$. 
The polynomial $\Psi(x,z)$ extends to  the unique polynomial $P(x_0:x_1; z_0: z_1)$ on  $\bC P^1_x\times \bC P^1_z$ of 
some bidegree $(n_x,n_z)$, where $n_x$ and $n_z$  are non-negative integers such that $n\le n_x+n_z\le 2n$. 

Now we have a well-defined map $\overline \Theta_{T,n} : \bC P^1_z \to Div(n_x)$, where $Div(n_x)$ is the space of 
positive divisors of degree $n_x$ on $\bC P^1_x$, which  associates to any point $z_0\in \bC P^1_z$ the 
positive divisor in $\bC P^1_x$ of degree $n_x$ obtained as the intersection of the compactification  $\overline \Ga_{T,n}\subseteq \bC P^1_x\times \bC P^1_z$ 
with the line $L_{z_0}$. The map $\overline\Theta_{T,n}$ straightforwardly extends to the map of subsets of $\bC P^1$ by 
taking the union of the supports of the image divisors for all points in a given set. In particular,  $\overline\Theta_{T,n}$ 
sends closed subsets of $\bC P^1$ to closed subsets. 
 The purpose of this section is to discuss some properties of  Hutchinson invariant sets for a given exactly 
 solvable operator $T$.  (In some literature maps sending points to subsets are called \defin{multivalued maps} or \defin{multifunctions}.)

\medskip
Observe that the existence of non-trivial (i.e. different from $\bC$) Hutchinson 
invariant sets for non-degenerate operators $T$ is implied by 
the following statement proven in \cite[Corollary 3.7]{AlBrSh1}.
 
 \begin{proposition}
If $T$ is a non-degenerate differential operator, then there is an integer $N_0$ and 
a positive number $R_0$ such that $D(0,R)$ is $T_{\ge n}$-invariant  whenever $n \geq N_0$ and $R \geq R_0$. 
\end{proposition}
\begin{proof} 
Observe that $T_{\ge n}$-invariance is stronger than  Hutchinson invariance in degree $n$ which, in particular, 
implies that the disk $D(0,R)$ is Hutchinson-invariant in degree $n$ as well. 
\end{proof} 

For the existence of non-trivial $\minvset{H,n}$, see \cite[Proposition 4.5]{Hemmingsson2024}. 

%
%

\subsection{Special classes of exactly solvable operators} 

\subsubsection{Classical Hutchinson IFS} 
Recall the following well-known definition. 

\begin{definition}
Let $X$ be a metric space. A finite collection $F=\{\phi_j\}_{j=1}^n$ of maps $\phi_j:X\to X$ is called an \defin{iterated function system}, 
(\defin{IFS} for short) and for a set $S \subset\bC$,  
\[
F(S)=\bigcup_{j=0}^n \phi_j(S).
\]
\end{definition} 

Historically, one mainly considered the case when $\phi_j:\bC\to \bC$ are linear contractions as in, e.g.,~\cite{Hutchinson1981}. 
In this situation there exists a unique \defin{attractor} of $F$ which is the unique proper subset of $\bC$ such that $F(S)=S$. 
Substantial research has been done in  this area, see e.g.~\cite{Barany2021} and  references therein.
We here extend the above notion of IFS's to  the case when the $\phi_j$ are allowed to be multivalued maps. We shall that the IFS is \emph{defined} by the (mutlivalued) maps $\phi_j$. 
Then the next two statements are straight-forward.

\begin{proposition}\label{prop:reduc} 
If for a given linear differential operator $T$, the associated bivariate polynomial $\Psi_{T,n}(x,z)=T[(x-z)^n]$ is 
reducible, i.e., $\Psi_{T,n}=\Psi_1\cdot \Psi_2\cdot \dotsm \cdot \Psi_\ell$ with irreducible $\Psi_j$'s, then a subset of $\bC$ is 
Hutchison-invariant sets in degree $n$ if and only if it is forward-invariant under the IFS defined by the multivalued maps $\theta_1,\dotsc, \theta_\ell$, where $\theta_j,\;j=1,\dots, \ell$ comes from the irreducible algebraic curve $\Gamma_j\subset\bC^2$ given by $\Psi_j(x,z)=0$.
\end{proposition}

\begin{theorem}\label{thm:hutchinson}
Suppose that $f_1,\dotsc,f_k$ are affine maps $f_j : \bC \to \bC$ of the 
form $f_j(z) = a_j z + b_j$, with $|a_j| < 1$ and let $M\subseteq \bC$ be the unique compact set
such that $M = \cup_j f_j(M)$. 
Then there is a unique exactly solvable non-degenerate 
operator $T$ of order at most $k$ for which $M=\minvset{H,n}$.
\end{theorem}
\begin{proof}
This follows from \cref{prop:Psi}, see also \cite{Hutchinson1981} for the existence and uniqueness of $M$.
\end{proof}

%
%

\subsubsection{$\Psi_{T,n}$ depends linearly on one of the variables}\label{sec:5.2}

In this case iterations of the algebraic curve given by $\Psi_{T,n}(x,z)=0$ essentially coincide 
with either forward or backward iterations of rational functions which is the 
most classical situation in complex dynamics. 

Namely, assume that $\Psi_{T,n}(x,z)=xU(z)-V(z)$, 
where $U(z)$ and $V(z)$ are polynomials that have no roots in common and that at least one of them has degree at least 2.
In this case, for $z\in \bC$, $\Theta_{T,n}(z)=\{\frac{V(z)}{U(z)}\}\setminus \{\infty\}$. 
In particular, the plane Julia set of $\frac{V(z)}{U(z)}$ is $T_{H,n}$-invariant.
Notice that there can be many disjoint invariant sets. For example, 
we can find two disjoint sets of periodic points of $\frac{V(z)}{U(z)}$ that are both $T_{H,n}$-invariant.
\medskip 

If $\Psi_{T,n}(x,z)=zU(x)-V(x)$ with the same assumptions on $U(x)$ and $V(x)$, then we get the 
backwards iteration of the same rational function. Let $V(\infty)=z_0$. If $z_0\neq \infty$ or $V^{-1}(z_0)\setminus \{\infty\}\neq \{z_0\}$, one can show that the only minimal sets are the plane Julia set of $\frac{V(x)}{U(x)}$ and possibly two subsets of the \emph{set of exceptional points} $\mathcal E$, where $|\mathcal E|\leq 2$, see \cref{sec:julia} (see also e.g. \cite{Beardon2000} for the definition of the set of exceptional points).

\subsubsection{$\Psi_{T,n}$ factorizes into polynomial factors each of which depend linearly on one (and the same) variable}\label{sec:5.3}

Assume that $\Psi_{T,n}(x,z)=(xU_1(z)-V_1(z))(xU_2(z)-V_2(z))\dots (xU_\ell(z)-V_\ell(z))$, where for each $i$, $V_i(z)$ and $U_i(z)$ have no common roots. 
Consider the $\ell$-tuple of rational functions $(\phi_1(z),\dots, \phi_\ell(z))$, where $\phi_i(z)=\frac{V_i(z)}{U_i(z)}:\hat \bC\to\hat \bC$.
Then, 
\[
\Theta_{T,n}(z)=\bigcup_{i=1}^\ell \{\phi_i(z)\}\setminus \{\infty\}.
\]
If instead $\Psi_{T,n}(x,z)=(zU_1(x)-V_1(x))(zU_2(x)-V_2(x))\dots (zU_\ell(x)-V_\ell(x))$, then 
\[
\Theta_{T,n}(z)=\bigcup_{i=1}^\ell \phi_i^{-1}(z)\setminus \{\infty\}.
\]

\subsubsection{$\Psi_{T,n}$ is independent of one of the variables}

Assume that $\Psi_{T,n}(x,z)=Q(x)$, where $Q(x)$ is  
some univariate polynomial of degree $n$. Then the corresponding differential operator is 
given by $T=\frac{1}{n!}Q(x)\frac{d^n}{dx^n}$.  Any non-empty subset of $\bC$ is mapped onto the 
zero locus of $Q(x)$ and $\minvset{H,n}$ coincides with this zero locus. 

Now assume that $\Psi_{T,n}(x,z)=Q(z)$, where $Q(z)$ is some univariate polynomial of degree $n$. 
Then the corresponding differential operator is given 
by $T=Q(x)\frac{d^n}{dx^n}-Q^\prime(x)\frac{d^{n-1}}{dx^{n-1}}+Q^{\prime\prime}(x)\frac{d^{n-2}}{dx^{n-2}}-\dots +(-1)^n Q^{(n)}(x)$.  
Each polynomial $(x-z_0)^n$ where $z_0$ is not a root of $Q(z)$ is mapped to a non-vanishing constant 
while if $z_0$ is a root of $Q(z)$ then it is mapped to $0$. In particular, each subset of $\bC$ is $T_{H,n}$-invariant.

\section{Differential operators and Julia sets}\label{sec:julia}

In this section, we drop the assumption that $T$ is exactly solvable and  completely characterize  $\minvset{1}$.
Since we are only interested in the images of linear polynomials $x-z$, $z\in \bC$,
the restriction of any linear differential operator $T$ with polynomial coefficients can be written as
\begin{equation}\label{eq:linearOperator}
T: (x-z) \mapsto Q_1(x) + Q_0(x)(x-z), 
\end{equation}
where $Q_1(x)$ and $Q_0(x)$ are the coefficients at $\frac{d}{dx}$ and the constant term of $T$, respectively.  
Equivalently, $T$ in \eqref{eq:linearOperator} can also be described via the relation 
\[
 T: (x-z) \mapsto z U(x) -  V(x) 
\]
where $V(x) = -(Q_1(x) + x Q_0(x))$, and $U(x) = -Q_0(x)$.

We utilize the notation $\hat\bC=\mathbb C P_{\infty}^1=\bC\cup\{\infty\}. $ Recall the following classical definition.

\begin{definition}
Given a rational function $R : \hat\bC \to \hat\bC$ of degree at least $2$,  its \emph{plane Julia set} $\mathcal J_\bC(R)$  
is defined as follows. First, we define the Fatou set $\mathcal F(R)$ of the rational function $R$ as the maximal open set on which $\{R^k(z):k\geq 0\}$ forms a normal family.
The Julia set $\mathcal J(R)$ of $R$ is the complement of $\mathcal F(R)$. The plane Julia set of $\mathcal J_\bC(R)$ is $\mathcal J(R)\cap \bC$. 
\end{definition}

\medskip
\begin{lemma}[See e.g. \cite{Falconer2004,Beardon2000}]\label{lem:juliaProperties}
The Julia set $\mathcal J(R)$ is a completely invariant set under $R$,
that is $ \mathcal J(R)= R(\mathcal J(R)) = R^{-1}(\mathcal J(R))$.
Furthermore, the Julia set is the smallest, closed, completely invariant subset of $\hat\bC$ containing at least $|\mathcal E|+1$ point(s), where $\mathcal E$ is the set of exceptional points.
\end{lemma}

Given our setup, we need the following additional definition. 
\begin{definition}\label{def:exc}
A rational function $R:\hat\bC\to \hat\bC$ of degree at least $2$ is called \defin{non-exceptional} if  the set of
exceptional points $\mathcal E$ of $R$ is a subset of $\{\infty\}$ and either $R^{-1}(R(\infty))\setminus \{\infty\}\neq \{R(\infty)\}$ or $R(\infty)= \infty$.

\end{definition}
The reason we need this definition is that we study the invariant subsets of $\bC$. If one instead would projectivize and work in $\hat\bC$, one could drop the second assumption in \cref{def:exc} and change the former condition to $|\mathcal E|=0$. This is also true for the special cases of exactly solvable operators considered in Sections~\ref{sec:5.2}-\ref{sec:5.3}.
When saying that $R$ is non-exceptional, we shall implicitly assume that its degree is at least 2.

\medskip
We have the following proposition.

\begin{proposition}\label{prop:juliaset}

Suppose that
\begin{equation}\label{eq:juliaOperator}
T: (x-z) \mapsto z U(x)-V(x),
\end{equation}
$U(x)$ and $V(x)$ have no common zeros, and that $R(x)\coloneqq  V(x)/U(x)$ is non-exceptional. Then  $\minvset{1} \in \invset{1}$
coincide with the plane Julia set $\mathcal J_\bC(R)$.
\end{proposition}
\begin{proof}
Fix $S\in \invset{1}$ and $z_0\in S$. By assumption on $R$, we can find a natural number $N$ such that $z_{j}\in R^{-1}(z_{j-1})$ and $z_j\neq \infty$ for each $j=1,\dots,N$ and $\infty\notin \bigcup_{j=0}^\infty R^{-j}(z_N)$. Now, $T(x-z) = V(x) -z U(x)$,
and therefore, all roots $x\in \bC$ of $z = V(x)/U(x)$
belongs to $S$.
But this is equivalent to saying that $R^{-1}(\alpha)\setminus \{\infty\} \subseteq S$, with $R$ defined as above, in particular $z_1\in S$. Iteration of this argument gives that $z_N\in S$ and $\bigcup_j R^{-j}(z_N) \subseteq S$.
Since $\mathcal J(R)$ is contained in the closure of $\bigcup_j f^{-j}(z_N)$, we obtain that every closed invariant set $S$ contains $\mathcal J_\bC(R)$.

On the other hand, by \cref{lem:juliaProperties}, $J_\bC(R)$ is invariant under $R^{-1}$, 
so $ \mathcal J_\bC(R) \in \invset{1}$. It follows that $ \mathcal J_\bC(R) $ is unique minimal closed set in $\invset{1}$.
\end{proof}

\begin{corollary}\label{cor:juliaSetEquivalence}
Let $T(p) = Q_1(x)p' + Q_0(x)p$, and suppose that $R(x) \coloneqq \frac{V(x)}{U(x)}$ is non-exceptional where $V(z):=(-Q_1(x)+xQ_0(x))$ and $U(x):=-Q_0(x)$. Assume additionally, that $V(x)$ and $U(x)$  have no common zeros,
Then the set $\minvset{1}$ coincides with the plane Julia set of $R$.
Conversely, for any non-exceptional rational 
function $R$, its plane Julia set can be realized as $\minvset{1}$ for an appropriate linear differential operator $T$.
\end{corollary}

\begin{example}\label{example:twoMinimalInvariantSets}
The differential operator $T = x(x-1)\frac{d}{dx} + 1$
admits \emph{two} minimal sets for $n=1$: either singleton $\{0\}$ or the unit circle.
This is in line with the aforementioned statements; the rational function $\frac{x(x-1)+x}{1}=x^2$ has non-exceptional set $\mathcal E$ equal to $\{0,\infty\}$.
\end{example}

\medskip

The following proposition shows that invariant sets of higher degrees are 
also closely related to Julia sets.
\begin{proposition}\label{prop:juliaInclusion}
Given a linear differential operator  $T = \sum_{j=0}^k Q_j(x)\frac{d}{dx}$,  assume that $S \in \invset{n}$ and fix $\alpha\in S$.
Define polynomials $U_{\alpha}(x)$ and $V_{\alpha}(x)$ via the relation
\[
T((x-\alpha)^{n-1}(x-z)) = zU_{\alpha}(x) -  V_{\alpha}(x)
\]
 and set  $R_\alpha(x) \coloneqq \frac{V_{\alpha}(x)}{U_{\alpha}(x) }$. If $R_\alpha$ is non-exceptional and $V_{\alpha}(x)$ and $U_{\alpha}(x) $ have no common zeros, then the plane Julia set $J_\bC(R_\alpha)$ is a subset of $S$.
\end{proposition}

\begin{proof}
Repeat the proof of \cref{prop:juliaset}.
\end{proof}
Roughly speaking, $S$ contains a family of plane Julia sets parametrized by $\alpha \in S$.

\medskip 

The following simple proposition provides a different inclusion of minimal invariant sets.
\begin{proposition}\label{lem:m1m2Inclusion}
For the operators $T_1 = Q_1(x)\frac{d}{dx} + Q_0(x)$ and $T_2 = \frac{1}{2}Q_1(x)\frac{d}{dx} + Q_0(x)$,
we have that $\minvsetT{H,1}{T_1} = \minvsetT{H,2}{T_2}$
and $\minvsetT{1}{T_1} \subseteq \minvsetT{2}{T_2}$, provided that $\minvsetT{1}{T_1}$ and $\minvsetT{1}{T_1} $ exist, respectively.
\end{proposition}
\begin{proof}
This follows from the fact that $T_2[(x-\alpha)^2] =(x-\alpha)T_1[(x-\alpha)]$.
\end{proof}

\begin{figure}[!ht]
\centering
\begin{subfigure}[b]{0.45\textwidth}
\includegraphics[width=\textwidth]{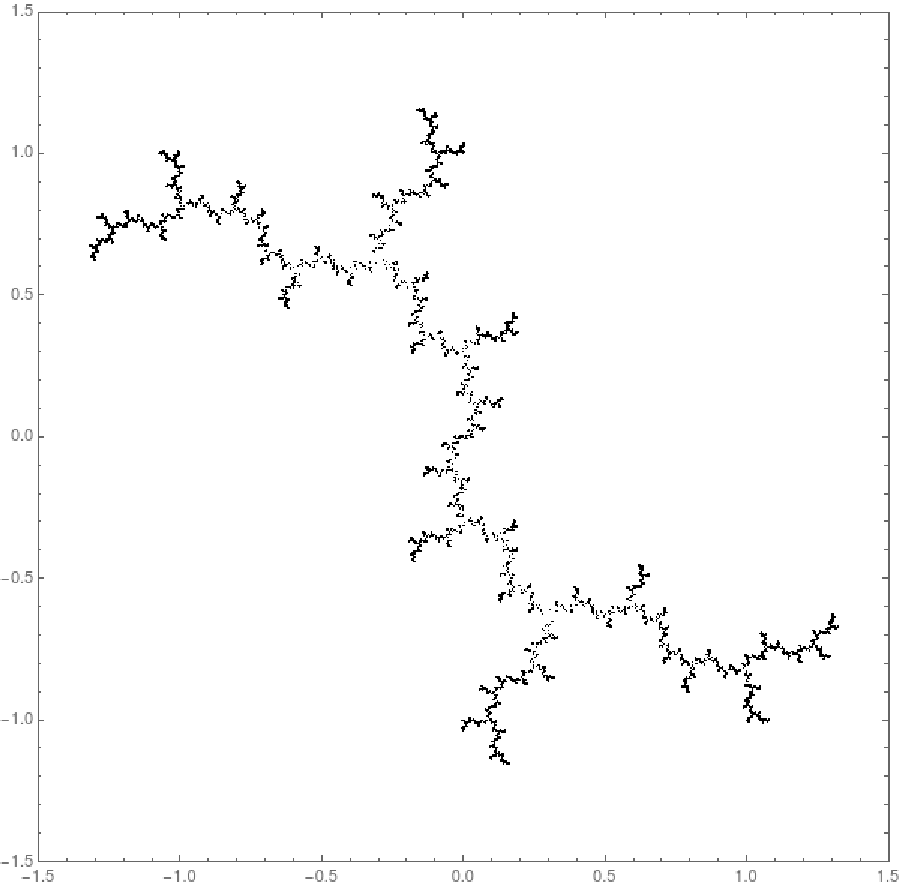}
\caption{$\minvsetT{1}{T_1}$}\label{fig:julia}
\end{subfigure}%
\begin{subfigure}[b]{0.45\textwidth}
\includegraphics[width=\textwidth]{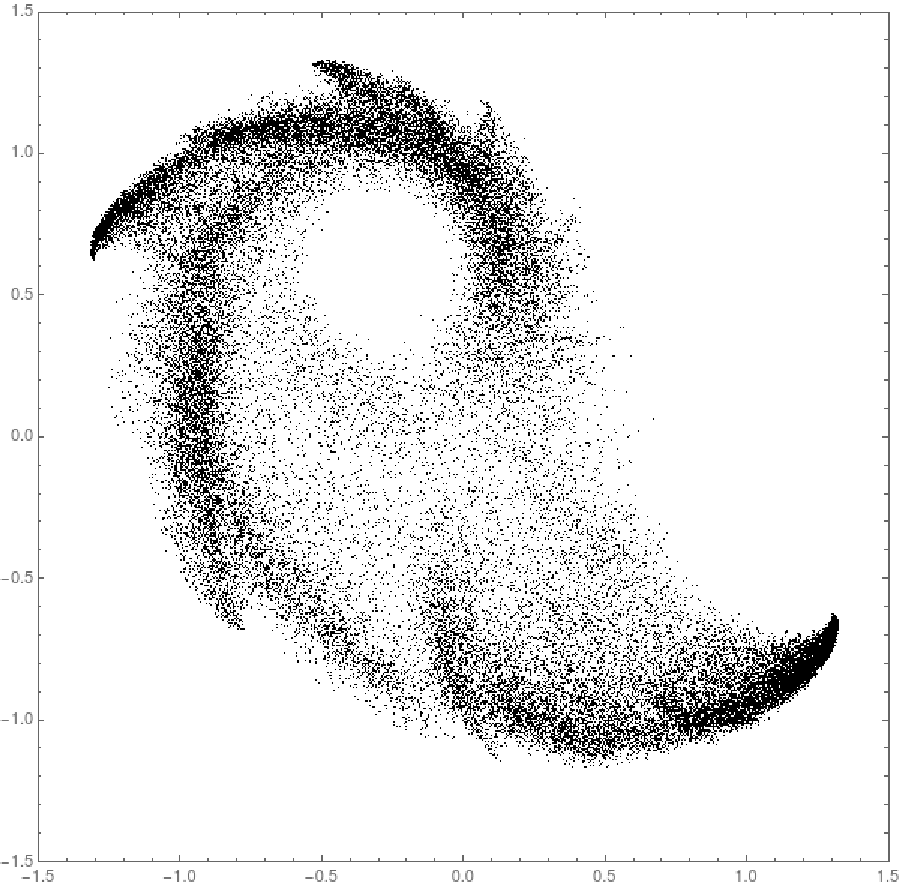}
\caption{$\minvsetT{2}{T_2}$}\label{fig:juliab}
\end{subfigure}
\caption{
The minimal sets $\minvsetT{1}{T_1}$ and $\minvsetT{2}{T_2}$ for $T_1 = (x^2-x+i)\frac{d}{dx} + 1$ and  $T_2 = \frac{1}{2}(x^2-x+i)\frac{d}{dx} + 1$. }\label{fig:degreeTwoJulia}
\end{figure}


\section{Examples}\label{sec:examples}

We start this section with an example illustrating \cref{lem:m1m2Inclusion}.

\begin{example}\label{ex:ex71}
The differential operator $T_1 = (x^2-x+i)\frac{d}{dx} + 1$
gives the classical Julia set associated with $f(x) = x^2+i$ as the unique minimal set $\minvsetT{1}{T_1}$, see \cref{fig:julia}.
For the differential operator $T_2 = \frac{1}{2}(x^2-x+i)\frac{d}{dx} + 1$, we have that $\minvsetT{2}{T_2}$
is given by \cref{fig:juliab}.
Note that $\minvsetT{1}{T_1}$ is a subset of $\minvsetT{2}{T_2}$ as implied by \cref{lem:m1m2Inclusion}.
The seemingly empty region in \cref{fig:juliab} is an artefact of the computer algorithm used 
to generate the figure described in \cref{subsec:forwardIteration}.
\end{example}

\begin{example}
If $$T= \left(\frac{x^3}{6}-\frac{x^2}{6}+\frac{i x}{6}\right)\frac{d^3}{dx^3}+ \left(\frac{x^2}{6}+\frac{i}{6}\right)\frac{d^2}{dx^2} +\frac{1}{3}\frac{d}{dx},$$
then $T[(x-z)^3)=\Psi_{T,3}(x,z)= (z-(x^2+i))(z-2x)$, and $\Theta_{T,3}$ equals the IFS defined by $x \to x/2$ and $x \to \{\pm \sqrt{x^2-i}\}$. $\minvset{H,3}$ is illustrated in \cref{fig:modifiedJulia}.
The multivalued mao $x \to \{\pm \sqrt{x^2-i}\}$ have minimal invariant set equal to the Julia set of $x^2+i$,
a ``dendrite'', see \cref{fig:julia} and this Julia set is thus contained in $\minvset{H,3}$. 
\begin{figure}
\centering
  \includegraphics[width=0.8\textwidth]{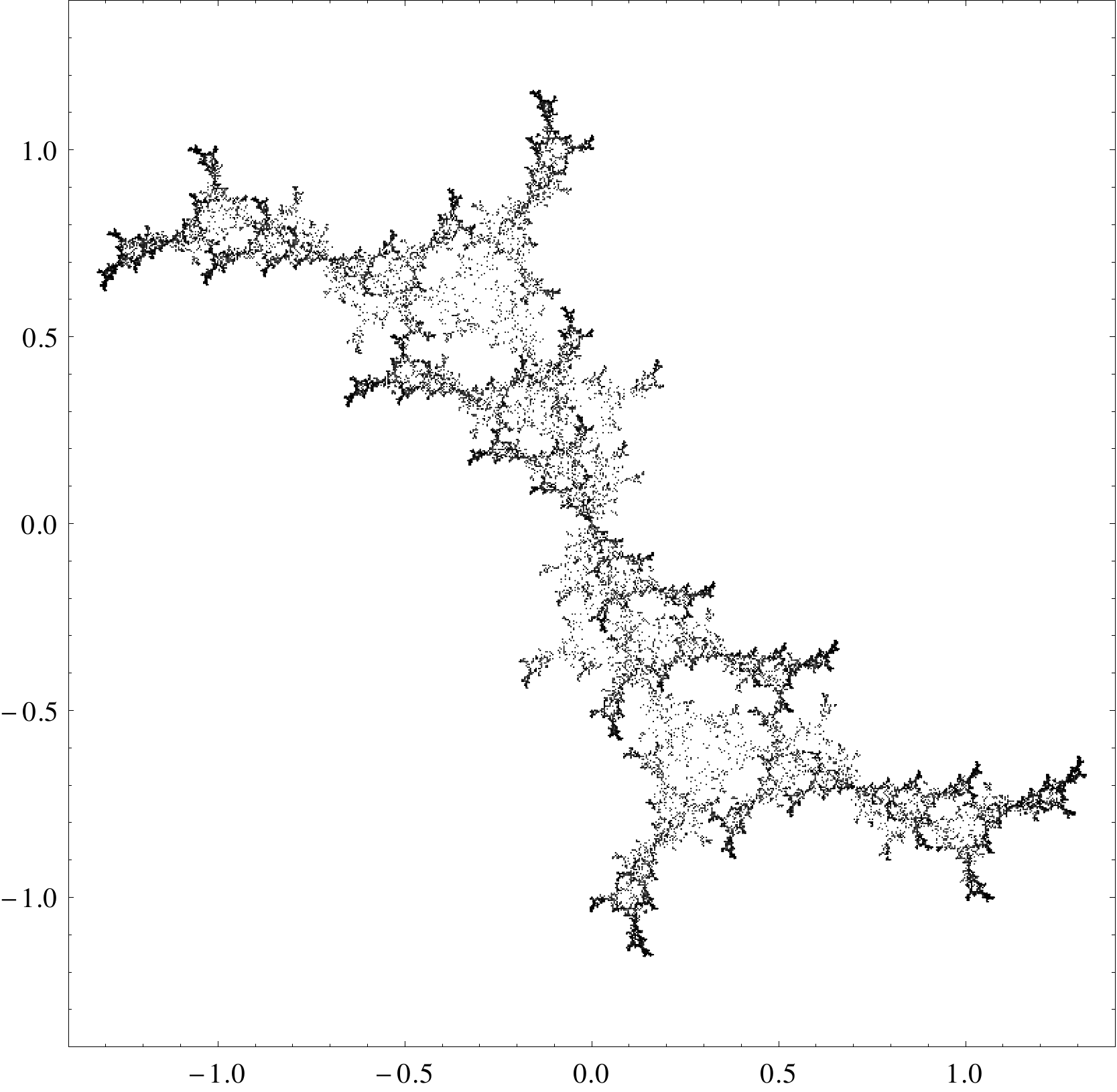}
\caption{$\minvset{H,n}$ for $\Psi(x,z)= (z-(x^2+i))(z-2x)$. It contains the Julia set of $x^2+i$.}\label{fig:modifiedJulia}
\end{figure}
\end{example}

\begin{example}
Set
\[
\delta=(x^2-1)\frac{d^2}{dx^2}+2x\frac{d}{dx}.
\]
The eigenpolynomials of $\delta$ are the Legendre polynomials
\[
 \{P_n\}_{n=0}^\infty= \{1,x,\frac{3x^2-1}{2},\dots\},
\]
satisfying $\delta(P_n) = n(n+1)P_n$. 
It is well-known that the zeros of the Legendre polynomials are located in $[-1, 1]$. Hence for any polynomial $Q(x) = \sum_{k=0}^N a_k x^k$,
 a finite order differential operator
\[
 T = \sum_{k=0}^Na_k\delta^{\circ k}
\]
also has $\{P_n\}^\infty_{n=0}$ as its eigenpolynomials but with the  eigenvalues $\{Q(n(n+1))\}_{n=0}^\infty$. 
Choose $Q$ (and thus $T$ ) such that $Q(0) > Q(6) > 0$. For example, 
  $Q(x)=1+(x-6)^2$. Then
\[
  T(3x^2/2)=T(P_2+P_0/2)=Q(6)\frac{3x^2 -1}{2} +Q(0)/2 =Q(6)\frac{3x^2}{2} + \frac{Q(0)-Q(6)}{2}
\]
has non-real zeros. Thus $T$ does not preserve the property of having all zeros in an interval, 
even though $T$ has eigenpolynomials with all zeros in an interval.
\end{example}

\begin{example}\label{ex:hutchinson}
 In the case $n=2$, one can easily show that  the exactly solvable non-degenerate operator $T$ satisfies 
 \begin{equation}\label{eq:split}
T[(x-z)^2] = (x-(a_1z+b_1))(x-(a_2z+b_2))
\end{equation}
 if and only if  $Q_0$, $Q_1$ and $Q_2$ are given as
\begin{align*}
Q_0 &= a_1 a_2, \\  
Q_1 &= -\frac{1}{2}\left( (2a_1a_2 - a_1-a_2)x + a_1 b_2 + a_2b_1 \right), \\
Q_2 &= \frac{1}{2}\left( 
(a_1-1)(a_2-1)x^2 + (a_1 b_2 + a_2b_1 - b_1 - b_2)x + b_1b_2 \right).
\end{align*}

If we choose $a_1 = -9/4$, $a_2 = 3/8$, $b_1=b_2=1$,
we have that the first eigenvalues of $T$ are
\[
 \lambda_0 = -27/32, \qquad \lambda_1 = -15/16, \qquad \lambda_2 = 1
\]
and
\[
 T = \left( \frac{65 x^2}{64}-\frac{31 x}{16}+\frac{1}{2} \right)\frac{d^2}{dx^2} + 
  \left( -\frac{3 x}{32} + \frac{15}{16} \right)\frac{d}{dx} 
-\frac{27}{32}.
\]
By \eqref{eq:split}, any $T_{H,n}$-invariant set $S\subseteq \bC$ also contains $a_1S+1$ as well as $a_2S+1$. Hence, as $|a_1|>1$ the only bounded $T_{H,n}$-invariant set $S$ might be $\{4/13\}$, since $4/13$ is the unique fixed point of $z\to a_1z+1$. However, such $S$ also contains $a_2S+1$, but $4a_2/13+1\neq 4/13$, which implies that no  bounded $T_{H,n}$-invariant set exists.

This example shows that  eigenvalues with distinct (and increasing) absolute values as in \cref{th:generalN} cannot guarantee
the boundedness of $\minvset{n}$. Namely,  one can generalize the above example using the product of several linear factors and as soon as one of them is expanding then no bounded sets will exist, but the behaviour of the eigenvalues can be more or less arbitrary. 
\end{example}

Note that $T_{H,n}$-invariant sets include classical fractals such as the Sierpi{\'{n}}ski triangle,
the middle thirds Cantor set, the L{\'e}vy curve (see \cref{ex:levy}) and the Koch curve.

\begin{example}\label{ex:levy}
For the differential operator $T = x(x+1)\frac{d^2}{dx^2} + i \frac{d}{dx} + 2$, we have that $\minvset{H,2}$
is a L{\'e}vy curve. It is straightforward to show that the roots of $T[(x-\alpha)^2]$ are given by 
\[
 \frac{1+i}{2}\alpha \qquad \text{ and }\qquad \frac{1-i}{2}(\alpha - i).
\]
The two maps 
\[
 \alpha \mapsto \frac{1+i}{2}\alpha \qquad \text{ and } \qquad \alpha \mapsto \frac{1-i}{2}(\alpha - i)
\]
are both contracting maps which together produce a fractal L{\'e}vy curve as shown in \cref{fig:levy}.
Both maps are indeed contracting affine maps, so in particular, every $T_{H,2}$-invariant set 
must contain the fixpoints of these maps. In particular, we can be sure that $\minvset{2}$ is non-empty,
as it must be a superset of $\minvset{H,2}$.
\end{example}

%

\section{Final Remarks} \label{sec:final}

Here is a very small sample of open problems related to the topic of this paper. 

\medskip

We have the following conjecture. 
\begin{conjecture}\label{conj:mobius}  
Let $T: \bC_n[x]  \to \bC_n[x]$ be an invertible linear operator and $S \in \invset{n}$,
such that $S$ contains at least $2$ points and has finite cardinality.
Then 
\[
 T[f] = (cx+d)^n f \left( \frac{ax+b}{cx+d}  \right)
\]
where $\phi(x) \coloneqq (ax+b)/(cx+d)$ is a M\"obius map such that $\phi^{k}(x)=x$
for some $k\geq 1$. 

In the special case that $T$ sends all degree $n$ polynomials to degree $n$ polynomials, then $T$ is induced by a rotation by a rational angle. That is, after a linear change of variables, we have 
\[
T\left[\prod_{i=1}^n(z-a_i)\right]=\prod_{i=1}^n(\alpha z-a_i)
\]
for some root of unity  $\alpha$. 
\end{conjecture}
The second part of \cref{conj:mobius} would be an improvement of \cref{prop:cardinality 2}.

\begin{remark} 
Observe that by \cite[Theorem 2.10.1]{JonesSingerman1987} every M\"obius transformation of 
finite order is conjugate to $z\mapsto cz$, where $c$ is a root of unity.
\end{remark}

\smallskip

\begin{problem} 
Study the uniqueness property for $\minvset{n}$ in case when several eigenvalues have the same maximal absolute value.
\end{problem}

\begin{problem} 
What can be said about the dimensional properties of $\minvset{n}$?
\end{problem}

\begin{problem} 
When is $\minvset{n}$ bounded?
\end{problem}

\begin{problem}
How does $\minvset{n}$ depend on $T$? 
What are the regions of continuous dependence of $\minvset{n}$ (in the Hausdorff metric) on $T$ in the parameter space?
\end{problem}

\begin{problem}
For $T$ exactly solvable, when is $\bigcap_n \minvset{n}$ non-empty?
\end{problem}

\bibliographystyle{amsalpha}
\bibliography{theBibliography}

\end{document}